\documentclass{article}

\usepackage{amsmath,amssymb,amsfonts,amsthm}

\newtheorem{theorem}{Theorem}[section]
\newtheorem{lemma}[theorem]{Lemma}
\newtheorem{proposition}[theorem]{Proposition}

\numberwithin{equation}{section}

\theoremstyle{remark}
\newtheorem{remark}[theorem]{Remark}

\newcommand{\grad}{\mathop{\mathrm{grad}}}
\newcommand{\II}{\mathop{\mathrm{II}}}
\newcommand{\Ric}{\mathop{\mathrm{Ric}}}
\newcommand{\tens}{\mathop{\mathrm{tens}}}
\newcommand{\YM}{\mathop{\mathrm{YM}}}
\newcommand{\dist}{\mathop{\mathrm{dist}}\nolimits}
\newcommand{\Ad}{\mathop{\mathrm{Ad}}}
\newcommand{\End}{\mathop{\mathrm{End}}}
\newcommand{\trace}{\mathop{\mathrm{trace}}}
\newcommand{\norm}{\mathop{\mathrm{norm}}}
\newcommand{\e}{\mathrm{e}}

\title{The Li-Yau-Hamilton estimate and the Yang-Mills heat equation on manifolds with boundary}

\author{Artem Pulemotov}

\begin{document}

\maketitle

\begin{center}
Department of Mathematics, Cornell University, \\ 310 Malott Hall,
Ithaca, NY 14853-4201, USA \\ E-mail: artem@math.cornell.edu
\end{center}

\begin{abstract}
The paper pursues two connected goals. Firstly, we establish the
Li-Yau-Hamilton estimate for the heat equation on a manifold $M$
with nonempty boundary. Results of this kind are typically used to
prove monotonicity formulas related to geometric flows. Secondly, we
establish bounds for a solution~$\nabla(t)$ of the Yang-Mills heat
equation in a vector bundle over $M$. The Li-Yau-Hamilton estimate
is utilized in the proofs. Our results imply that the curvature
of~$\nabla(t)$ does not blow up if the dimension of~$M$ is less
than~4 or if the initial energy of~$\nabla(t)$ is sufficiently
small.
\end{abstract}

\section{Introduction}

The present paper considers two related subjects. Section~\ref{sec
LYH} establishes the Li-Yau-Hamilton estimate for the heat equation
on a manifold with boundary. Results of this kind are known to be
useful in the study of geometric flows. Sections~\ref{sec YMH}
and~\ref{sec exit time} discuss estimates for the solutions of the
Yang-Mills heat equation in a vector bundle over a manifold with
boundary. The proofs utilize a probabilistic technique. Our results
imply that the curvature of a solution does not blow up if the
dimension of the manifold is less than~4 or if the initial energy is
sufficiently small.

The Li-Yau-Hamilton estimate for the heat equation generalizes the
well-known differential Harnack inequality of~\cite{PLSTY86}. This
estimate was originally obtained on manifolds without boundary in
the paper~\cite{RH93a}. It is typically used to prove monotonicity
formulas related to various geometric evolution equations; see, for
example,~\cite{RH93b}. In their turn, such monotonicity formulas are
essential for establishing the existence of solutions.

Let us mention that~\cite{BCRH97} offers a constrained version of
the Li-Yau-Hamilton estimate from~\cite{RH93a}. The
paper~\cite{HDCLN05} adapts the result of~\cite{RH93a} to K\"ahler
manifolds. We point out that an inequality similar to the
Li-Yau-Hamilton estimate for the heat equation comes up in the
investigation of the Ricci flow. Its precise formulation and various
applications are presented in~\cite[Chapter~15]{BC_etal_unfinished}.
Analogous results hold for the K\"ahler-Ricci flow. Their
formulations and relevant references can be found
in~\cite[Chapter~2]{BC_etal07} and in~\cite{LN07}.

Suppose $M$ is a smooth compact Riemannian manifold without
boundary. Consider a positive solution $p(t,x)$ to the heat equation
on $M$ such that the integral $\int_Mp(t,x)\,\mathrm{d} x$ does not
exceed~1 for any $t\in(0,\infty)$. Then there exist constants $A>0$
and $B>0$ that depend only on the manifold $M$ and satisfy
\begin{align}\label{intro1}
D^2_{\cdot,\cdot}\log
p(t,x)&\ge-\left(\frac1{2t}+A\left(1+\log\left(\frac B{t^\frac
{\dim M}2p(t,x)}\right)\right)\right)\langle\cdot,\cdot\rangle,\nonumber \\
t&\in(0,1],~x\in M.
\end{align}
In this formula, $D^2_{\cdot,\cdot}$ is the second covariant
derivative, and $\left<\cdot,\cdot\right>$ is the Riemannian metric.
The inequality is to be understood in the sense of bilinear forms.
If~$M$ is Ricci parallel and has nonnegative sectional curvatures,
then~(\ref{intro1}) holds with $A=0$. This is the case when $M$ is,
for example, a sphere or a flat torus. Formula~(\ref{intro1})
constitutes the Li-Yau-Hamilton estimate for the heat equation. It
was originally obtained in~\cite{RH93a}.

Suppose now that $M$ is a smooth compact Riemannian manifold with
nonempty boundary $\partial M$. Section~\ref{sec LYH} of the present
paper establishes formula~(\ref{intro1}) in this case. The solution
$p(t,x)$ of the heat equation is assumed to satisfy the Neumann
boundary condition. Theorem~\ref{theorem LYH no crv}
proves~(\ref{intro1}) in the situation where no restrictions are
imposed on the curvature of $M$ away from $\partial M$. But the
boundary of $M$ must be totally geodesic for this result to hold.
Moreover, several derivatives of the curvature of $M$ have to vanish
at~$\partial M$. Theorem~\ref{theorem LYH crv} deals with a more
exclusive situation. It shows that inequality~(\ref{intro1}) holds
with $A=0$ if the manifold $M$ is Ricci parallel and has nonnegative
sectional curvatures. As before, $\partial M$ must be totally
geodesic. However, the previously mentioned derivatives of the
curvature of $M$ are no longer required to vanish at~$\partial M$.
Our proofs of Theorems~\ref{theorem LYH no crv} and~\ref{theorem LYH
crv} differ considerably in their techniques.

Both incarnations of estimate~(\ref{intro1}) appearing in
Section~\ref{sec LYH} play significant roles in establishing the
results of Section~\ref{sec YMH}. More precisely, they enable us to
obtain a monotonicity formula related to the Yang-Mills heat
equation. This formula is given by Lemma~\ref{lemma monot}. It helps
us establish an estimate for the solutions to the Yang-Mills heat
equation in dimensions~5 and higher.

In order to prove Theorem~\ref{theorem LYH no crv}, we employ the
doubling method. More precisely, we consider two identical copies of
$M$ and glue them together along the boundary. This procedure
produces a closed manifold $\mathcal M$. The desired estimate
follows by applying the results of the paper~\cite{RH93a} on
$\mathcal M$. Of course, several technical questions need to be
handled in order to make the doubling method work for our purpose.

The proof of Theorem~\ref{theorem LYH crv} relies on the Hopf
boundary point lemma for vector bundle sections appearing
in~\cite{Artem}. The technique we use resembles those employed
in~\cite{PLSTY86,RH93a}. One may also apply the doubling method to
prove Theorem~\ref{theorem LYH crv}. However, the approach adopted
in the present paper appears to be more effective. Firstly, it
enables us to avoid the assumption on the curvature of $M$ near
$\partial M$ that is required to carry out the doubling procedure.
Secondly, it does not rely on the previously known versions of the
Li-Yau-Hamilton estimate. Last but not least, our approach seems to
be more natural and to provide a better ground for further
generalizations.

Section~\ref{sec YMH} of the present paper deals with the Yang-Mills
heat equation in a vector bundle over a compact Riemannian manifold
$M$ with nonempty boundary. In order to describe our results, we
need to outline the setup. Let $E$ be a vector bundle over $M$.
Suppose the time-dependent connection $\nabla(t)$ in $E$ solves the
Yang-Mills heat equation
\begin{equation}\label{intro2}
\frac\partial{\partial
t}\nabla(t)=-\frac12\mathrm{d}^*_{\nabla(t)}R^{\nabla(t)},\qquad
t\in[0,T).
\end{equation}
Here and in what follows, $\mathrm{d}_{\nabla(t)}$ is the exterior
covariant derivative, $\mathrm{d}^*_{\nabla(t)}$ is its adjoint, and
$R^{\nabla(t)}$ is the curvature of $\nabla(t)$. By definition,
$R^{\nabla(t)}$ is a 2-form on $M$ with its values in the
endomorphism bundle $\End E$. The Yang-Mills heat equation is a
potentially powerful instrument for minimizing the Yang-Mills energy
functional; see, for example,~\cite{MARB82,JR92,MARBAT02}. It has a
number of applications in topology and in mathematical physics. Some
of these applications are comprehensively discussed in the
book~\cite{SDPK90} and the dissertation~\cite{LS87}; see
also~\cite{ZBMHLSCL87}. The existence of solutions is one of the
most important questions regarding the Yang-Mills heat equation.

Since $\partial M$ is assumed to be nonempty, we have to specify the
boundary conditions for the time-dependent connection~$\nabla(t)$.
Doing so is a delicate matter. As detailed in Remark~\ref{rem BCR vs
BCdel}, it is more natural for us to impose the boundary conditions
on the curvature $R^{\nabla(t)}$ than on $\nabla(t)$ itself. We
assume
\begin{equation}
\label{intro3}
\left(R^{\nabla(t)}\right)_{\tan}=0,\qquad\left(\mathrm{d}^*_{\nabla(t)}R^{\nabla(t)}\right)_{\tan}=0,\qquad
t\in[0,T).
\end{equation}
The subscript ``$\tan$" stands for the component of the
corresponding $\End E$-valued form that is tangent to $\partial M$.
Alternatively, we may assume
\begin{equation}\label{intro4}
\left(R^{\nabla(t)}\right)_{\norm}=0,\qquad\left(\mathrm{d}_{\nabla(t)}R^{\nabla(t)}\right)_{\norm}=0,\qquad
t\in[0,T).
\end{equation}
(Actually, the second equality always holds due to the Bianchi
identity.) The subscript ``$\norm$" signifies the component that is
normal to $\partial M$. Conditions~(\ref{intro3}) and~(\ref{intro4})
are analogous to the relative and the absolute boundary conditions
for real-valued forms. The results in Section~\ref{sec YMH} prevail
regardless of whether we choose~(\ref{intro3}) or~(\ref{intro4}) to
hold on $\partial M$. Other ways to introduce the boundary
conditions in the context of Yang-Mills theory were considered in
several works including, for
example,~\cite{AM92,JSGS94,JTJS97,WEG06,LGNCpre}. We should mention,
however, that none of these works except~\cite{LGNCpre} deals with
parabolic-type equations like~(\ref{intro2}). The relationship
between the boundary conditions utilized in the present paper and
the boundary conditions appearing elsewhere is discussed in
Remark~\ref{rem BC relation}.

Section~\ref{sec YMH} provides estimates for the curvature
$R^{\nabla(t)}$ of the solution $\nabla(t)$ to the Yang-Mills heat
equation~(\ref{intro2}) subject to~(\ref{intro3}) or~(\ref{intro4}).
Roughly speaking, we show that $R^{\nabla(t)}$ is bounded at every
point of $M$ by expressions involving the initial energy of
$\nabla(t)$. Theorem~\ref{thm YMH dim23} considers the case where
the dimension of $M$ is either~2 or~3. It yields an estimate on
$R^{\nabla(t)}$ and demonstrates that $R^{\nabla(t)}$ does not blow
up. Theorem~\ref{thm YMH dim4} deals with the case where the
dimension is equal to~4. It requires that the initial energy of
$\nabla(t)$ be smaller than a constant depending on~$M$. If this
assumption is satisfied, the theorem produces a bound on
$R^{\nabla(t)}$. It is easy to see that $R^{\nabla(t)}$ does not
blow up when this bound holds. Theorem~\ref{thm YMH dim5+} considers
the situation where the dimension of $M$ is greater than or equal
to~5. It produces an estimate on $R^{\nabla(t)}$ under a rather
sophisticated condition. The theorem implies that the curvature of a
solution to Eq.~(\ref{intro2}) cannot blow up after time~$\rho$ if
the initial energy is smaller than a number depending on $\rho$.

When the dimension of $M$ equals~2, 3, or~4, the boundary $\partial
M$ has to be convex for the results in Section~\ref{sec YMH} to
hold. No other assumptions on the geometry of $M$ are required.
However, if the dimension is~5 or higher, the situation is
different. In this case, $\partial M$ has to be totally geodesic,
and restrictions have to be imposed on the curvature of $M$. The
reason for such a phenomenon lies in the fact that, when the
dimension is~5 or higher, our arguments involve the Li-Yau-Hamilton
estimate~(\ref{intro1}). Both Theorems~\ref{theorem LYH no crv}
and~\ref{theorem LYH crv} are exploited.

We thus observe a trichotomy in the behavior of the
solution~$\nabla(t)$ to Eq.~(\ref{intro2}). Theorems~\ref{thm YMH
dim23}, \ref{thm YMH dim4}, and~\ref{thm YMH dim5+} provide three
different sets of conditions ensuring that $R^{\nabla(t)}$ does not
blow up. Each of these sets corresponds to a certain range of
dimensions of~$M$. A similar trichotomy occurs on closed manifolds;
see, for instance,~\cite{MARBAT02}. However, the difference in the
geometric assumptions that was discussed in the previous paragraph
is not observed in this case.

Let us make a comment as to the practical importance of the results
in Section~\ref{sec YMH}. Proving that the curvature does not blow
up is the principal ingredient in establishing the long-time
existence of solutions to the Yang-Mills heat equation. The list of
relevant references includes but is not limited
to~\cite{SDPK90,JR92,MS94,MARBAT02,LGNCpre}. We should point out
that all these works except~\cite{LGNCpre} restrict their attention
to manifolds without boundary.

The proofs of Theorems~\ref{thm YMH dim23}, \ref{thm YMH dim4},
and~\ref{thm YMH dim5+} rely on the probabilistic technique
developed in~\cite{MARBAT02}. The origin of this technique lies in
the theory of harmonic maps; see~\cite{AT96}. The pivotal stochastic
process in our considerations is a reflecting Brownian motion on the
manifold $M$. Let us mention that the probabilistic approach to
Yang-Mills theory was investigated rather extensively. The
paper~\cite{MARBAT02} contains a series of results and a list of
references on the subject.

While establishing the theorems in Section~\ref{sec YMH}, we prove a
noteworthy property of $\End E$-valued forms on $M$. The precise
phrasing of this property is given by Lemma~\ref{lemma Neumann
norm}. Roughly speaking, it states that, if $\partial M$ is convex
and an $\End E$-valued form $\phi$ satisfies~(\ref{intro3})
or~(\ref{intro4}), then the derivative of the squared absolute value
of $\phi$ in the direction of the outward normal to $\partial M$
must be nonpositive. A simpler version was established
in~\cite{LGNCpre}.

Section~\ref{sec exit time} of the present paper provides an exit
time estimate for a reflecting Brownian motion on a manifold with
convex boundary. This result helps us prove another inequality for
the curvature of the connection~$\nabla(t)$ discussed above.

\section{The Li-Yau-Hamilton estimate}\label{sec LYH}

Consider a smooth, compact, connected, oriented, $n$-dimensional
Riemannian manifold $M$ with nonempty boundary $\partial M$. We
suppose $n\ge2$. This section aims to study the solutions of the
heat equation on $M$ with the Neumann boundary condition. More
precisely, we will obtain two versions of the Li-Yau-Hamilton
estimate for such solutions.

The Riemannian curvature tensor will be designated by $R(X,Y)Z$ when
applied to the vectors $X$, $Y$, and $Z$ from the tangent space
$T_xM$ at the point $x\in M$. We use the usual notation
\[R(X,Y,Z,W)=\left<R(X,Y)Z,W\right>,\qquad X,Y,Z,W\in T_xM.\]
The angular brackets with no lower index refer to the scalar product
in the space $T_xM$ given by the Riemannian metric. The Ricci tensor
will be written as $\Ric(X,Y)$ when applied to $X,Y\in T_xM$. We
will impose substantial assumptions on the curvature of $M$ in
Theorem~\ref{theorem LYH crv} below.

The Levi-Civita connection $D$ in the tangent bundle $TM$ induces
connections in the tensor bundles over $M$. We preserve the notation
$D$ for all of them. Our further arguments require introducing
higher-order differential operators. Let us describe the
corresponding procedure. Fix a tensor field $T$ and two or more
vector fields $Y_1,\ldots,Y_k$ on $M$. Set $D_{Y_1}^1T$ equal to
$D_{Y_1}T$. We define the $k$th covariant derivative
$D^k_{Y_1,\ldots,Y_k}T$ inductively by the formula
\[D^k_{Y_1,\ldots,Y_k}T=D_{Y_k}\left(D^{k-1}_{Y_1,\ldots,Y_{k-1}}T\right)
-\sum_{i=1}^{k-1}D^{k-1}_{Y_1,\ldots,Y_{i-1},D_{Y_k}Y_i,Y_{i+1},\ldots,Y_{k-1}}T.\]
One can verify that the value of $D^k_{Y_1,\ldots,Y_k}T$ at the
point $x\in M$ does not depend on the values of $Y_1,\ldots,Y_k$
away from $x$.

Let $\nu$ be the outward unit normal vector field on $\partial M$.
The differentiation of real-valued functions in the direction of
$\nu$ will be denoted by $\frac\partial{\partial\nu}$. If the point
$x$ lies in $\partial M$, then the space $T_xM$ contains the
subspace $T_x\partial M$ tangent to $\partial M$. We write
$\II(X,Y)$ for the second fundamental form of $\partial M$ applied
to $X,Y\in T_x\partial M$. By definition,
$\II(X,Y)=\left<D_X\nu,Y\right>$. Some of the statements below
require that $\partial M$ be totally geodesic. In this case,
$\II(X,Y)=0$ for all $X,Y\in T_x\partial M$ at every point
$x\in\partial M$.

Suppose the smooth positive function $p(t,x)$ defined on
$(0,\infty)\times M$ solves the heat equation
\begin{equation}\label{heat eq p}
\left(\frac\partial{\partial t}-\Delta_M\right)p(t,x)=0, \qquad
t\in(0,\infty),~x\in M,
\end{equation}
with the Neumann boundary condition
\begin{equation}\label{Neumann BC p}
\frac\partial{\partial\nu}p(t,x)=0, \qquad t\in(0,\infty),~x\in
\partial M.
\end{equation}
The notation $\Delta_M$ represents the Laplace-Beltrami operator on
$M$. It should be mentioned that Theorem~\ref{theorem LYH no crv}
and Remark~\ref{rem p<Ct^n2} below assume the inequality
$\int_Mp(t,x)\,\mathrm{d} x\le1$ for all $t\in(0,\infty)$. Here and
in what follows, the integration over a Riemannian manifold is to be
carried out with respect to the Riemannian volume measure on the
manifold.

We are now in a position to formulate the first result of this
section. It establishes a general version of the Li-Yau-Hamilton
estimate for the function~$p(t,x)$.

\begin{theorem}\label{theorem LYH no crv}
Let the boundary $\partial M$ be totally geodesic. Suppose the
following statements hold:
\begin{enumerate}
\item\label{assu curv bndy}
The covariant derivative
$\left(D^k_{\nu,\ldots,\nu}R\right)(\nu,X,\nu,Y)$ is equal to~0 for
all positive odd $k$ and all $X,Y\in T_xM$ at every point
$x\in\partial M$.
\item\label{assu integral}
The integral $\int_Mp(t,x)\,\mathrm{d} x$ of the solution $p(t,x)$
to the boundary value problem~(\ref{heat eq p})--(\ref{Neumann BC
p}) does not exceed~1 at any $t\in(0,\infty)$.
\end{enumerate}
Then there exist constants $A>0$ and $B>0$ independent of $p(t,x)$
such that the estimate
\begin{equation}\label{LYH no curv}D^2_{X,X}\log p(t,x)\ge-\left(\frac1{2t}
+A\left(1+\log\left(\frac B{t^\frac
n2p(t,x)}\right)\right)\right)\langle X,X\rangle
\end{equation}
holds for every $t\in(0,1]$, $x\in M$, and $X\in T_xM$. (Recall that
$n$ is the dimension of~$M$.)
\end{theorem}

Conceptually, the proof consists in doubling $M$ to get a manifold
without boundary and exploiting the results of~\cite{RH93a}. A few
technical aspects need to be handled. The most essential problem is
to make sure the function to which we apply the theorem
in~\cite{RH93a} possesses the necessary differentiability
properties.

\begin{proof}
Let $\mathcal M$ be the double of $M$. More precisely, $\mathcal M$
appears as the quotient $(M\times\{1,2\})/\sim$~. The equivalence
relation $\sim$ is given as follows: Two distinct pairs, $(x,i)$ and
$(y,j)$, satisfy $(x,i)\sim(y,j)$ if and only if $x$ coincides with
$y$ and lies in $\partial M$. We preserve the notation $(x,i)$ for
the equivalence class of $(x,i)\in M\times\{1,2\}$. As described
in~\cite{HM91}, $\mathcal M$ carries the canonical smooth structure.
One may also obtain this structure by using Theorem~5.77
in~\cite{JWBRBL95} and the diffeomorphism $\mu(r,x)$ defined below.
We explain further in the proof how to introduce a local coordinate
system around $(x,i)\in\mathcal M$ when $x\in\partial M$. Note that
$\mathcal M$ is a manifold without boundary. The map $\mathcal
E_i(x)$ taking $x\in M$ to $(x,i)\in\mathcal M$ is an embedding for
both $i=1$ and $i=2$.

The Riemannian metric on $M$ induces a Riemannian metric on
$\mathcal M$ in a natural fashion. More precisely, the scalar
product $\left<X,Y\right>_{\mathcal M}$ of the vectors $X,Y\in
T_{(x,i)}\mathcal M$ is given by the formula
$\left<X,Y\right>_{\mathcal M}=\left<(\mathrm{d}\mathcal
E_i)^{-1}X,(\mathrm{d}\mathcal E_i)^{-1}Y\right>$. It is not
difficult to verify that $\left<\cdot,\cdot\right>_{\mathcal M}$ is
well-defined at every $(x,i)\in\mathcal M$. The proposition
in~\cite{HM91}, along with Assumption~\ref{assu curv bndy} of our
theorem, implies that $\left<\cdot,\cdot\right>_{\mathcal M}$
depends smoothly on $(x,i)\in\mathcal M$.

Introduce a positive function $\tilde p(t,z)$ on $(0,\infty)\times
\mathcal M$ by setting $\tilde p(t,(x,i))=\frac12p(t,x)$. Its
integral over the manifold $\mathcal M$ is bounded by 1. Our next
goal is to demonstrate that $\tilde p(t,z)$ solves the heat equation
on $\mathcal M$. This would allow us to apply the results
of~\cite{RH93a} and obtain estimate~(\ref{LYH no curv}) for this
function. Theorem~\ref{theorem LYH no crv} would then follow as a
direct consequence.

First and foremost, we need to prove that $\tilde p(t,z)$ is twice
continuously differentiable in the second variable. Consider the set
$\mathcal M^\partial\subset\mathcal M$ equal to $\mathcal
E_1(\partial M)$. Of course, this set is also equal to $\mathcal
E_2(\partial M)$. Using the smoothness of the function $p(t,x)$ on
$M$, one can easily establish the smoothness of $\tilde p(t,z)$
outside of $\mathcal M^\partial$. In consequence, it suffices to
show that $\tilde p(t,z)$ is twice continuously differentiable in a
neighborhood of an arbitrarily picked point $\tilde z\in\mathcal
M^\partial$.

There exists a unique $\tilde x\in\partial M$ satisfying $\tilde
z=\mathcal E_1(\tilde x)=\mathcal E_2(\tilde x)$. We need to
introduce local coordinates in $M$ around $\tilde x$. Suppose
$\epsilon>0$ is small enough to ensure that the mapping $\mu(r,x)$
defined on $[0,\epsilon)\times\partial M$ by the formula
$\mu(r,x)=\exp_x(-r\nu)$ is a diffeomorphism onto its image. The
existence of such an $\epsilon>0$ is justified
in~\cite[Chapter~11]{JMJS74}. Fix a coordinate neighborhood
$U^\partial$ of $\tilde x$ in the boundary $\partial M$ with a local
coordinate system $y_1,\ldots,y_{n-1}$ in~$U^\partial$ centered
at~$\tilde x$. Define the set $U$ as the image of
$[0,\epsilon)\times U^\partial$ under $\mu(r,x)$. Clearly, $U$ is a
neighborhood of $\tilde x$ in $M$. We extend $y_1,\ldots,y_{n-1}$ to
a coordinate system $x_1,\ldots,x_n$ in $U$ by demanding that the
equalities
\begin{align*}x_k(\mu(r,x))&=y_k(x),~x_n(\mu(r,x))=r,\\
r&\in[0,\epsilon),~x\in U^\partial,~k=1,\ldots,n-1,
\end{align*}
hold true; cf.~\cite{HM91}. Importantly, $\frac\partial{\partial
x_i}$ is tangent to the boundary on $U^\partial$ for every
$i=1,\ldots,n-1$. The vector field $\frac\partial{\partial x_n}$
coincides with $-\nu$ on this set.

The coordinate system $x_1,\ldots,x_n$ in $U$ gives rise to a
coordinate system $z_1,\ldots,z_n$ in the neighborhood $\mathcal
U=\mathcal E_1(U)\cup\mathcal E_2(U)$ of~$\tilde z$. Namely, suppose
$z\in\mathcal U$ equals $\mathcal E_i(x)$ with $x\in U$. Define
$z_k(z)=x_k(x)$ when $k=1,\ldots,n-1$ and $z_n(z)=(-1)^{i+1}x_n(x)$.
We will now analyze the partial derivatives of $\tilde p(t,z)$ with
respect to the newly introduced local coordinates. By doing so, we
will establish the desired differentiability properties of this
function.

It is easy to understand that $\frac{\partial}{\partial z_k}\tilde
p(t,z)$ exists and coincides with $\frac12\frac{\partial}{\partial
x_k}p(t,x)$ if $z=(x,i)\in\mathcal U$ and $k=1,\ldots,n-1$.
Furthermore, $\frac{\partial}{\partial z_k}\tilde p(t,z)$ is
continuous on $\mathcal U$ for these $k$. The situation is slightly
more complicated when we differentiate with respect to the last
coordinate. A straightforward argument shows
\[\frac{\partial}{\partial z_n}\tilde p(t,z)=\frac{(-1)^{i+1}}2\frac{\partial}{\partial x_n}p(t,x)\] when
$z=(x,i)\in\mathcal U\setminus\mathcal M^\partial$. The one-sided
derivatives $\frac{\partial^+}{\partial z_n}\tilde p(t,z)$ and
$\frac{\partial^-}{\partial z_n}\tilde p(t,z)$ coincide with
$\frac12\frac{\partial}{\partial x_n}p(t,x)$ and
$-\frac12\frac{\partial}{\partial x_n}p(t,x)$, respectively, if
$z=(x,i)\in\mathcal M^\delta$. The boundary condition~(\ref{Neumann
BC p}) ensures that $\frac{\partial}{\partial z_n}\tilde p(t,z)$ is
well-defined and equal to 0 on $\mathcal M^\partial$. We conclude
that $\frac{\partial}{\partial z_n}\tilde p(t,z)$ exists in
$\mathcal U$. Furthermore, it is continuous on $\mathcal U$.

Let us turn our attention to the second derivatives. Analogous
reasoning can be used here. The existence and the continuity of
$\frac{\partial^2}{\partial z_k\partial z_l}\tilde p(t,z)$ on
$\mathcal U$ are clear for $k=1,\ldots,n-1$ and $l=1,\ldots,n$. In
order to analyze $\frac{\partial^2}{\partial z_n\partial z_k}\tilde
p(t,z)$ with $k=1,\ldots,n-1$, observe that the formula
\[
\frac{\partial^+}{\partial z_n}\frac{\partial}{\partial z_k}\tilde
p(t,z)=\frac12\frac{\partial^2}{\partial x_n\partial
x_k}p(t,x)=\frac12\frac{\partial^2}{\partial x_k\partial
x_n}p(t,x)=0
\]
holds when $z=(x,i)\in\mathcal M^\delta$. A similar calculation
suggests the equality $\frac{\partial^-}{\partial
z_n}\frac{\partial}{\partial z_k}\tilde p(t,z)=0$ on $\mathcal
M^\delta$. As a consequence, $\frac{\partial^2}{\partial z_n\partial
z_k}\tilde p(t,z)$ is well-defined and continuous on $\mathcal U$.
The same can be said about $\frac{\partial^2}{\partial z_n^2}\tilde
p(t,z)$. Indeed, the formula
\[
\frac{\partial^2}{\partial z_n^2}\tilde
p(t,z)=\frac{(-1)^{2i+2}}2\frac{\partial^2}{\partial
x_n^2}p(t,x)=\frac12\frac{\partial^2}{\partial x_n^2}p(t,x)
\] holds when $z=(x,i)\in\mathcal U$.

Summarizing the arguments above, we arrive at the following verdict:
The function $\tilde p(t,z)$ is twice continuously differentiable in
$z$ on the manifold $\mathcal M$. The smoothness of $\tilde p(t,z)$
in $t$ is evident. With this in mind, one can readily verify that
the heat equation
\begin{equation}\label{heat mathcal M}
\left(\frac\partial{\partial t}-\Delta_\mathcal M\right)\tilde
p(t,z)=0, \qquad t\in(0,\infty),~z\in\mathcal M,\end{equation} is
satisfied ($\Delta_\mathcal M$ denoting the Laplace-Beltrami
operator on $\mathcal M$). In addition, the integral of $\tilde
p(t,z)$ over $\mathcal M$ is bounded by~1. These observations enable
us to apply Theorem~4.3 of~\cite{RH93a}. As a result, we get the
existence of constants $\tilde A>0$ and $\tilde B>0$ such that
\[
\tilde D^2_{X,X}\log\tilde p(t,z)\ge-\left(\frac1{2t} +\tilde
A\left(1+\log\left(\frac{\tilde B}{t^\frac n2\tilde
p(t,z)}\right)\right)\right)\langle X,X\rangle\] for every
$t\in(0,1]$, $z\in\mathcal M$, and $X\in T_z\mathcal M$. Here,
$\tilde D^2_{X,X}$ refers to the second covariant derivative given
by the Levi-Civita connection in~$T\mathcal M$. Inequality~(\ref{LYH
no curv}) follows immediately with $A=\tilde A$ and $B=2\tilde B$.
 \end{proof}

\begin{remark}
As in the proof of Theorem~\ref{theorem LYH no crv}, let $\mathcal
M$ be the double of the manifold~$M$. Given $z\in\mathcal M$, the
tangent space $T_z\mathcal M$ carries a natural scalar product
induced by the Riemannian metric on $M$. This scalar product depends
smoothly on $z\in\mathcal M$ if and only if the boundary $\partial
M$ is totally geodesic and Assumption~1 of Theorem~\ref{theorem LYH
no crv} is fulfilled. The justification of this fact can be found
in~\cite{HM91}.
\end{remark}

\begin{remark}
Since the function $\tilde p(t,z)$ appearing in the proof
satisfies~(\ref{heat mathcal M}), it must be smooth on
$(0,\infty)\times\mathcal M$. In order to verify this, one may use
the uniqueness and the integral representation of solutions to the
heat equation; see, e.g.,~\cite[Proposition~4.1.2]{EH02b}.
\end{remark}

\begin{remark}
Estimate~(\ref{LYH no curv}) means that $D^2_{\cdot,\cdot}\log
p(t,x)$ is greater than or equal to
\[-\left(\frac1{2t}+A\left(1+\log\left(\frac B{t^\frac
n2p(t,x)}\right)\right)\right)\langle\cdot,\cdot\rangle\] in the
sense of bilinear forms for every $t\in(0,1]$ and $x\in M$.
\end{remark}

\begin{remark}\label{rem p<Ct^n2}
If Assumption~\ref{assu integral} of Theorem~\ref{theorem LYH no
crv} is fulfilled, then there exists a constant $C>0$ independent of
$p(t,x)$ such that
\begin{equation}\label{p<Ct^-n/2}
p(t,x)\le Ct^{-\frac n2},\qquad t\in(0,1],~x\in M.
\end{equation}
Note that $\partial M$ does not have to be totally geodesic for this
to hold. In the case where $p(t,x)$ tends to a delta function as $t$
tends to~0, formula~(\ref{p<Ct^-n/2}) follows from the parametrix
construction for the Neumann heat kernel. This observation was made
in~\cite[Proof of Lemma~3.2]{EH02a}. We also refer to~\cite{JW97}
for relevant results. In the general case, formula~(\ref{p<Ct^-n/2})
can be established by using the integral representation of the
solution to the heat equation; see,
e.g.,~\cite[Proposition~4.1.2]{EH02b}. Importantly, if all the
assumptions of Theorem~\ref{theorem LYH no crv} are fulfilled and
$C$ satisfies~(\ref{p<Ct^-n/2}), then there exists a constant
$A_C>0$ such that~(\ref{LYH no curv}) holds with $A=A_C$ and $B=C$.
\end{remark}

We now state a more specific version of the Li-Yau-Hamilton estimate
for the function~$p(t,x)$. It shows how~(\ref{LYH no curv})
simplifies when the appropriate curvature restrictions are imposed
on~$M$ away from the boundary. Note that the inequality
$\int_Mp(t,x)\mathrm{d} x\le1$ is no longer required for our
arguments.

\begin{theorem}\label{theorem LYH crv}
Let the boundary $\partial M$ be totally geodesic. Suppose the
following statements hold at every point $x\in M$:
\begin{enumerate}
\item\label{assu Ricci par}
The covariant derivative $(D_X\Ric)(Y,Z)$ is equal to 0 for all
$X,Y,Z\in T_xM$.
\item\label{assu nneg sect}
The sectional curvature of every plane in $T_xM$ is nonnegative.
That is, $R(X,Y,Y,X)\ge0$ for all $X,Y\in T_xM$.
\end{enumerate}
Then the solution $p(t,x)$ of the boundary value problem~(\ref{heat
eq p})--(\ref{Neumann BC p}) satisfies the inequality
\begin{equation}\label{LYH crv}
D^2_{X,X}\log p(t,x)\ge-\frac1{2t}\langle X,X\rangle\end{equation}
for every $t\in(0,\infty)$, $x\in M$, and $X\in T_xM$.
\end{theorem}

In many situations, estimate~(\ref{LYH crv}) can be established by
the same technique we used to establish Theorem~\ref{theorem LYH no
crv}. One just has to exploit Corollary~4.4 in~\cite{RH93a} instead
of Theorem~4.3 in~\cite{RH93a}. However, we prefer to adduce a
direct method of proving~(\ref{LYH crv}) here based on the Hopf
lemma for vector bundle sections; see~\cite{Artem}. Firstly, because
this method does not require the equality
$\left(D^k_{\nu,\ldots,\nu}R\right)(\nu,X,\nu,Y)=0$ to hold on
$\partial M$. Secondly, because it avoids using the results
of~\cite{RH93a}. Last but not least, we believe the direct method is
more illuminating and gives a more fertile ground for
generalizations.

\begin{proof}
Take a number $\epsilon>0$. Given $t\in[0,\infty)$, introduce the
two times covariant tensor field $L^\epsilon_t$ by the formula
\[L^\epsilon_t(X,Y)=(t+\epsilon)D^2_{X,Y}\log
p(t+\epsilon,x)+\frac12\langle X,Y\rangle, \qquad X,Y\in T_xM.
\] Our plan is to use the Hopf boundary point lemma
of~\cite{Artem} for showing that $L^\epsilon_t$ is positive
semidefinite at every point of $M$. The theorem will then be proved
by taking the limit as $\epsilon$ goes to 0.

In what follows, we assume $p(t,x)$ is defined and smooth on
$[0,\infty)\times M$. This does not lead to any loss of generality.
Indeed, we can always establish the desired estimate for the
function $p_\delta(t,x)=p(t+\delta,x)$, $\delta>0$, and pass to the
limit as $\delta$ tends to 0.

Firstly, let us compute $\left(\frac\partial{\partial
t}-\Delta_{\tens}\right)L^\epsilon_t$. The Laplacian
$\Delta_{\tens}$ in this expression appears as the trace of the
second covariant derivative $D^2$ in the bundle $T^*M\otimes T^*M$.
Recall that the connection in this bundle is induced by the
Levi-Civita connection in $TM$.

The Riemannian metric on $M$ yields a scalar product of tensors over
a point $x\in M$. The notation $\langle\cdot,\cdot\rangle$ is
preserved for this scalar product. Set $P^\epsilon(t,x)=\grad\log
p(t+\epsilon,x)$. We omit the $(t,x)$ at $P^\epsilon(t,x)$ when this
does not lead to ambiguity. Introduce the mapping $\Phi(t,w)$ acting
from $[0,\infty)\times(T^*_xM\otimes T^*_xM)$ to $T^*_xM\otimes
T^*_xM$ by the equality
\begin{align*}
\Phi(t,w)(X,Y)&=2\langle\mathcal
R_{X,Y},w\rangle-\langle\iota_X\Ric,\iota_Yw\rangle
-\langle\iota_Y\Ric,\iota_Xw\rangle
\\ &\hphantom{=}~+\frac2{t+\epsilon}\langle\iota_Xw,\iota_Yw\rangle
+2(t+\epsilon)R(X,P^\epsilon,P^\epsilon,Y) \\
&\hphantom{=}~-\frac1{t+\epsilon}w(X,Y),\qquad X,Y\in T_xM.
\end{align*}
Here, the tensor $\mathcal R_{X,Y}$ is defined as $\mathcal
R_{X,Y}(Z,W)=R(X,Z,W,Y)$ for $Z,W\in T_xM$, and $\iota$ denotes the
interior product. A standard calculation, together with
Assumption~\ref{assu Ricci par} of our theorem, shows that
\[
\left(\frac\partial{\partial
t}-\Delta_{\tens}\right)L^\epsilon_t=D_{2P^\epsilon}L^\epsilon_t+\Phi(t,L^\epsilon_t),\qquad
t\in[0,\infty),
\]
at every $x\in M$. For relevant arguments,
see~\cite{RH93a,BCRH97,HDCLN05} and~\cite[Section~2.5]{BCPLLN06}.

Let $W\subset T^*M\otimes T^*M$ be the set of two times covariant,
symmetric, positive semidefinite tensors. Suppose $\epsilon$ is
chosen sufficiently small to ensure that $L_t^\epsilon$ belongs to
$W$ at every point of $M$ when $t=0$. The existence of such an
$\epsilon$ follows from the smoothness of $p(t,x)$ on
$[0,\infty)\times M$. Fixing $T>0$, we will apply Theorem~2.1
in~\cite{Artem} (the Hopf lemma) to demonstrate that $L^\epsilon_t$
must belong to $W$ at every point of $M$ for all $t\in[0,T]$.

Some more notation has to be introduced here. Given $x\in M$, define
the set $W_x$ as the intersection of $W$ with $T_x^*M\otimes
T_x^*M$. Evidently, $W_x$ is closed and convex in $T^*_xM\otimes
T^*_xM$. Let $\omega(w)$ stand for the point in $W_x$ nearest to
$w\in T_x^*M\otimes T_x^*M$. More precisely, the minimum of the
scalar product $\langle w-v,w-v\rangle$ over $v\in W_x$ must be
attained at $v=\omega(w)$. Denote $\lambda(w)=w-\omega(w)$.

We now verify the assumptions of Theorem~2.1 from~\cite{Artem}. It
was already noted that $L_t^\epsilon\in W$ at every point of $M$
when $t=0$ and that $W_x$ was closed and convex for all $x\in M$.
The set $W$ is invariant under the parallel translation in
$T^*M\otimes T^*M$; see~\cite[The arguments preceding
Corollary~10.12]{BC_etal_unfinished}. The mapping $\Phi(t,w)$,
obviously, satisfies inequality~(2.1) in~\cite{Artem}. Thus,
Requirement~2 of Theorem~2.1 in that paper remains the only
statement to be checked. Considering Remark~2.1 of~\cite{Artem}, it
suffices to prove the inequality
\begin{equation}\label{point inward}\langle\Phi(t,\omega(L^\epsilon_t)),\lambda(L^\epsilon_t)\rangle\le0,\qquad t\in[0,T],\end{equation}
over every point of $M$.

Fix $t\in[0,T]$. We omit the subscript $t$ at $L^\epsilon_t$ in
order to simplify the notation. Pick an orthonormal basis
$\{e_1,\ldots,e_n\}$ of the space $T_x M$ for some $x\in M$. Without
loss of generality, suppose this basis diagonalizes $L^\epsilon$ at
$x$. One can easily understand that
\begin{align*}
\omega(L^\epsilon)(e_i,e_j)&=\max\{L^\epsilon(e_i,e_j),0\},
\\ \lambda(L^\epsilon)(e_i,e_j)&=\min\{L^\epsilon(e_i,e_j),0\},\qquad i,j=1,\ldots,n.
\end{align*}
Hence
\[\langle\Phi(t,\omega(L^\epsilon)),\lambda(L^\epsilon)\rangle=\sum_{i=1}^n
\Phi(t,\omega(L^\epsilon))(e_i,e_i)\min\{L^\epsilon(e_i,e_i),0\}.
\]
If $L^\epsilon(e_i,e_i)<0$, then $\omega(L^\epsilon)(e_i,e_j)=0$ for
all $j=1\ldots,n$. Using this fact along with our
Assumption~\ref{assu nneg sect}, one can readily prove that
\[\Phi(t,\omega(L^\epsilon))(e_i,e_i)\ge0
\]
when $L^\epsilon(e_i,e_i)<0$. Thus, estimate~(\ref{point inward})
holds true.

We are now in a position to apply Theorem~2.1 of~\cite{Artem}. More
precisely, we apply Corollary~2.3 of that theorem. Let us establish
the equality $\langle\lambda(L^\epsilon_t),D_\nu
L^\epsilon_t\rangle=0$ over an arbitrarily chosen point
$x\in\partial M$ for all $t\in[0,T]$. This would lead us to the
conclusion that $L^\epsilon_t$ is always positive semidefinite.

As before, we fix $t\in[0,T]$ and write $L^\epsilon$ instead of
$L^\epsilon_t$. Pick an orthonormal basis $\{v_1,\ldots,v_{n-1}\}$
of the space $T_x
\partial M$ tangent to the boundary. Suppose this basis diagonalizes
the restriction of $L^\epsilon$ to $T_x \partial M\otimes T_x
\partial M$. A straightforward verification shows
\[ L^\epsilon(v_i,\nu)=-(t+\epsilon)\II(v_i,P^\epsilon), \qquad i=1,\ldots,n-1.
\]
(Remark that $P^\epsilon$ is tangent to $\partial M$ due to the
Neumann boundary condition~(\ref{Neumann BC p}).) The right-hand
side of the above formula is equal to 0 because $\partial M$ is
totally geodesic. Hence $L^\epsilon(v_i,\nu)=0$ for
$i=1,\ldots,n-1$. We conclude that the orthonormal basis
$\{v_1,\ldots,v_{n-1},\nu\}$ diagonalizes $L^\epsilon$ at $x$ and
\begin{align}\label{bnd cond basis}
\langle\lambda(L^\epsilon),D_\nu
L^\epsilon\rangle=&\sum_{i=1}^{n-1}\min\{L^\epsilon(v_i,v_i),0\}\left(D_\nu
L^\epsilon\right)(v_i,v_i) \nonumber
\\ &+\min\{L^\epsilon(\nu,\nu),0\}\left(D_\nu
L^\epsilon\right)(\nu,\nu).
\end{align}

Each of the summands on the right-hand side of~(\ref{bnd cond
basis}) is 0. Indeed, since $\partial M$ is totally geodesic, we can
introduce the normal coordinates $x_1,\ldots,x_n$ around $x$ so that
$\frac\partial{\partial x_i}$ and $\frac\partial{\partial x_n}$
coincide with $v_i$ and $-\nu$, respectively, at the origin. A
calculation in these coordinates yields
\begin{align}\label{third der zero} \left(D_\nu
L^\epsilon\right)(v_i,v_i)=&-(t+\epsilon)(D_{v_i}(\iota_{P^\epsilon}\II))(v_i)
\nonumber \\ &-(t+\epsilon)\II(v_i,D_{v_i}P^\epsilon) \nonumber \\
&+(t+\epsilon)R(v_i,P^\epsilon,v_i,\nu),\qquad i=1,\ldots,n-1.
\end{align}
(The vector $D_{v_i}P^\epsilon$ is tangent to the boundary because
$\left<D_{v_i}P^\epsilon,\nu\right>=\frac1{t+\epsilon}\,L^\epsilon(v_i,\nu)=0$.)
The second fundamental form $\II$ vanishes identically. Therefore,
the first two terms in~(\ref{third der zero}) equal 0. Given
$X,Y,Z\in T_x\partial M$, it is easy to see that $R(X,Y)Z$ coincides
with the Riemannian curvature tensor of $\partial M$ applied to
these vectors. Hence $R(v_i,P^\epsilon)v_i$ is tangent to $\partial
M$, and the third term in~(\ref{third der zero}) equals 0, as well.
As a result, $(D_\nu L^\epsilon)(v_i,v_i)=0$ for $i=1\ldots,n-1$.

Another calculation (cf.~\cite{PLSTY86}) yields
\begin{align*}
(D_\nu
L^\epsilon)(\nu,\nu)&=(t+\epsilon)\frac\partial{\partial\nu}\Delta_M
\log p(t+\epsilon,x)-\sum_{i=1}^{n-1}\left(D_\nu
L^\epsilon\right)(v_i,v_i)
\\ &=(t+\epsilon)\frac\partial{\partial\nu}\Delta_M\log
p(t+\epsilon,x)=2(t+\epsilon)\II(P^\epsilon,P^\epsilon).
\end{align*} Since $\II$ vanishes identically, the above implies
$(D_\nu L^\epsilon)(\nu,\nu)=0$. In view of~(\ref{bnd cond basis}),
we conclude $\langle\lambda(L^\epsilon),D_\nu L^\epsilon\rangle$
equals 0 over our arbitrarily chosen $x\in\partial M$.

Corollary~2.3 of Theorem~2.1 in~\cite{Artem} now suggests that
$L^\epsilon_t$ is positive semidefinite at every point of $M$ for
all $t\in[0,T]$. Since no restrictions were imposed on the number
$T$, this tensor field must be positive semidefinite at every point
for all $t\in[0,\infty)$. Taking the limit as $\epsilon$ tends to 0
proves~(\ref{LYH crv}).  \end{proof}

\section{The Yang-Mills heat equation}\label{sec YMH}

This section aims to study the solutions to the Yang-Mills heat
equation in a vector bundle over the manifold $M$. Roughly speaking,
we show that the curvature of such a solution is bounded if the
dimension of~$M$ is less than~4 or if the initial energy is
sufficiently small. The proofs utilize a probabilistic method. When
the dimension of $M$ is greater than or equal to~5, our technique
requires the Li-Yau-Hamilton estimate established in
Section~\ref{sec LYH}. Notably, this reflects on the assumptions we
impose on the geometry of~$M$.

Many statements below demand that the boundary $\partial M$ be
convex. The concept of convexity is quite delicate for Riemannian
manifolds. Different definitions and the relations between them are
surveyed in~\cite{MS01}. The paper~\cite{SK79} is also relevant. In
what follows, when saying $\partial M$ is convex, we mean that the
formula \begin{equation}\label{def convex}\II(X,X)\ge0,\qquad X\in
T\partial M,\end{equation} must hold for the second fundamental form
of $\partial M$.

The next few paragraphs provide a description of the structure
required to formulate the Yang-Mills heat equation. For a detailed
exposition of the background material,
see~\cite{JPBBL81,BL85,JJ88,SDPK90,DFKU84}.

Recall that the manifold $M$ is assumed to be compact. Let $E$ be a
vector bundle over $M$ with the standard fiber $\mathbb R^d$ and the
structure group $G$. We suppose $G$ appears as a Lie subgroup of
$\mathrm O(d)$ and acts naturally on $\mathbb R^d$. The symbol
$\mathfrak g$ stands for the Lie algebra of $G$. In what follows, we
assume $\mathbb R^d$ is equipped with the standard scalar product.
Every element of $\mathfrak g$ appears as a skew-symmetric
endomorphism of $\mathbb R^d$. Define the scalar product in this Lie
algebra by the formula
\[
\left<A,B\right>_{\mathfrak g}=-\trace AB,\qquad A,B\in\mathfrak g.
\]
The adjoint bundle $\Ad E$, whose standard fiber is equal to
$\mathfrak g$, carries the fiber metric induced by~$\left<\cdot,
\cdot\right>_\mathfrak g$.

Let $\nabla$ be a connection in~$E$. We understand $\nabla$ as a
mapping that takes a section $\tau$ of $E$ to a section $\nabla\tau$
of the bundle $T^*M\otimes E$. It is customary to interpret
$\nabla\tau$ as an $E$-valued 1-form on the manifold $M$. Consider a
vector field $X$ on $M$. We write $\nabla_X\tau$ to indicate the
application of $\nabla\tau$ to $X$. Given a smooth real-valued
function $f(x)$ on $M$, the formula
\[
\nabla_X(f\tau)=(Xf)\tau+f\nabla_X\tau
\]
must be satisfied. We suppose $\nabla$ is compatible with the
structure group $G$. The curvature of $\nabla$ will be denoted by
$R^{\nabla}$. Let us mention that $R^{\nabla}$ appears as a 2-form
on $M$ with its values in the bundle $\Ad E$. Our goal is to write
down the Yang-Mills heat equation. In order to do this, we need to
introduce the operators of covariant exterior differentiation
corresponding to a connection in~$E$.

Consider the bundle $\Lambda^pT^*M\otimes\Ad E$ for a nonnegative
integer $p$. Its sections are interpreted as $\Ad E$-valued
$p$-forms on the manifold $M$. The set of all these sections will be
designated by $\Omega^p(\Ad E)$. The Riemannian metric on $M$ and
the fiber metric in $\Ad E$ give rise to a scalar product in the
fibers of $\Lambda^pT^*M\otimes\Ad E$. We use the notation
$\left<\cdot,\cdot\right>_E$ for this scalar product and the
notation $|\cdot|_E$ for the corresponding norm.

The connections $D$ in $TM$ and $\nabla$ in $E$ induce a connection
in the bundle $\Lambda^pT^*M\otimes\Ad E$. It appears as a mapping
from $\Omega^p(\Ad E)$ to the set of sections of
$T^*M\otimes\Lambda^pT^*M\otimes\Ad E$. We preserve the notation
$\nabla$ for this connection in $\Lambda^pT^*M\otimes\Ad E$. Define
the operator $\mathrm{d}_{\nabla}$ acting from $\Omega^p(\Ad E)$ to
$\Omega^{p+1}(\Ad E)$ by the formula
\begin{align*}
\left(\mathrm{d}_{\nabla}\phi\right)&(X_1,\ldots,X_{p+1}) \\ &=
\sum_{i=1}^{p+1}(-1)^{i+1}\left(\nabla_{X_i}\phi\right)(X_1,\ldots,
X_{i-1},X_{i+1},\ldots,X_{p+1}).
\end{align*}
Here, $\phi$ belongs to $\Omega^p(\Ad E)$, and $X_1,\ldots,X_{p+1}$
belong to $T_xM$ for some $x\in M$. It is easy to understand that
$\mathrm{d}_{\nabla}$ plays the role of the covariant exterior
derivative corresponding to $\nabla$. The operator
$\mathrm{d}_{\nabla}^*$ acting from $\Omega^{p+1}(\Ad E)$ to
$\Omega^p(\Ad E)$ is defined by the equality
\[
\left(\mathrm{d}_{\nabla}^*\psi\right)(X_1,\ldots,X_p) =
-\sum_{i=1}^n\left(\nabla_{e_i}\psi\right)(e_i,X_1,\ldots,X_p).
\]
Here, $\psi$ belongs to $\Omega^{p+1}(\Ad E)$, the vectors
$X_1,\ldots,X_p$ belong to $T_xM$ for some $x\in M$, and
$\{e_1,\ldots,e_n\}$ is an orthonormal basis of $T_xM$. We set
$\mathrm{d}^*_{\nabla}$ to be equal to zero on $\Omega^0(\Ad E)$. In
view of Lemma~\ref{lemma int parts} below, this operator may be
understood as the formal adjoint of $\mathrm{d}_{\nabla}$.

Fix a number $T>0$. Consider a connection~$\nabla(t)$ in~$E$
depending on $t\in[0,T)$. The parameter $t$ will be interpreted as
time. We require that $\nabla(t)$ be compatible with the structure
group $G$ for all $t\in[0,T)$. Suppose $\nabla(t)$ satisfies the
Yang-Mills heat equation
\begin{equation}\label{YM heat eq}
\frac\partial{\partial
t}\nabla(t)=-\frac12\mathrm{d}^*_{\nabla(t)}R^{\nabla(t)},\qquad
t\in[0,T).
\end{equation}
In particular, this connection must be once continuously
differentiable in $t\in[0,T)$. The factor~$\frac12$ appears in the
right-hand side because we want to achieve maximum conformity with
the probabilistic results employed below. In interpreting
$\frac\partial{\partial t}\nabla(t)$, one should remember that
$\nabla(t)$ lies, for each $t\in[0,T)$, in the linear space of
mappings taking sections of $E$ to sections of $T^*M\otimes E$. Our
next step is to specify the boundary conditions for~$\nabla(t)$.
Doing this is quite a delicate matter. We discuss some of the
nuances in Remarks~\ref{rem BCR vs BCdel} and~\ref{rem BC relation}
in the end of this section.

Every $\Ad E$-valued $p$-form $\phi\in\Omega^p(\Ad E)$ can be
decomposed into the sum of its the tangential component
$\phi_{\tan}$ and its normal component $\phi_{\norm}$ on the
boundary of $M$. Roughly speaking, $\phi_{\tan}$ coincides with the
restriction of $\phi$ to the vectors from $T\partial M$. If $\phi$
lies in $\Omega^0(\Ad E)$, then $\phi_{\tan}$ equals $\phi$ on
$\partial M$. We are now ready to impose the boundary conditions on
$\nabla(t)$. Assume the equalities
\begin{equation}\label{relative BC}
\left(R^{\nabla(t)}\right)_{\tan}=0,\qquad\left(\mathrm{d}^*_{\nabla(t)}R^{\nabla(t)}\right)_{\tan}=0
\end{equation} hold on $\partial M$ for all $t\in[0,T)$. One should view~(\ref{relative BC}) as a
version of the relative boundary conditions on real-valued forms;
see, for example,~\cite{DRIS71}. Alternatively, we may assume the
formulas
\begin{equation}\label{absolute BC}
\left(R^{\nabla(t)}\right)_{\norm}=0,\qquad\left(\mathrm{d}_{\nabla(t)}R^{\nabla(t)}\right)_{\norm}=0
\end{equation} hold on $\partial M$ for all $t\in[0,T)$. (Actually, the second one is always satisfied due to the Bianchi
identity.) These should be viewed as a version of the absolute
boundary conditions; again,~\cite{DRIS71} is a good reference. The
arguments in the present paper will prevail regardless of whether we
choose Eqs.~(\ref{relative BC}) or Eqs.~(\ref{absolute BC}) to hold
on $\partial M$. For other problems and techniques, however, only
one of the choices may be appropriate.

We should make an important comment at this point. In essence,
Eqs.~(\ref{relative BC}) and~(\ref{absolute BC}) are restrictions on
the curvature form~$R^{\nabla(t)}$. Another possible strategy is to
impose the boundary conditions directly on the
connection~$\nabla(t)$. We postpone a discussion of this issue until
after the proofs of our results; see Remarks~\ref{rem BCR vs BCdel}
and~\ref{rem BC relation}.

Introduce the function
\[
\YM(t)=\int_M\left|R^{\nabla(t)}\right|_E^2\mathrm{d} x\] for
$t\in[0,T)$. In accordance with the conventions of Section~\ref{sec
LYH}, the integration is to be carried out with respect to the
Riemannian volume measure on $M$. It is reasonable to call $\YM(t)$
the energy at time $t$. A standard argument involving
Lemma~\ref{lemma int parts} below shows that $\YM(t)$ is
non-increasing in $t\in[0,T)$; see~\cite{LGNCpre} and also, for
example,~\cite{JJ88,JR92,YCCLS94}.

We now state the main results of Section~\ref{sec YMH}. Our first
theorem concerns the lower-dimensional case. It offers a bound for
$R^{\nabla(t)}$ in terms of the initial energy $\YM(0)$ and
demonstrates that $R^{\nabla(t)}$ does not blow up at time $T$. In
what follows, the notation $R^{\nabla(t)}(x)$ refers to the
curvature of $\nabla(t)$ at the point $x\in M$.

\begin{theorem}\label{thm YMH dim23}
Let the dimension $\dim M$ equal~2 or~3. Suppose~$\partial M$ is
convex in the sense of~(\ref{def convex}). Then the solution
$\nabla(t)$ of Eq.~(\ref{YM heat eq}), subject to the boundary
conditions~(\ref{relative BC}) or~(\ref{absolute BC}), satisfies the
estimate
\begin{equation}\label{bound ex23}
\sup_{x\in
M}\left|R^{\nabla(\rho)}(x)\right|^2_E\le\max\left\{\frac{4\YM(0)}{\rho^2}\,,
\theta_1\e^{\theta_2\sqrt{\YM(0)}}\YM(0)\right\}
\end{equation}
for all $\rho\in(0,T)$. Here, $\theta_1>0$ and $\theta_2>0$ are
constants depending only on the manifold $M$.
\end{theorem}

A similar result can be obtained in dimension~4 provided that the
initial energy $\YM(0)$ is smaller than a certain value $\xi$. We
emphasize that $\xi$ depends on nothing but~$M$.

\begin{theorem}\label{thm YMH dim4}
Let the dimension $\dim M$ equal~4. Suppose the boundary $\partial
M$ is convex in the sense of~(\ref{def convex}). Then there exists a
constant $\xi>0$ depending only on the manifold~$M$ and satisfying
the following statement: The solution $\nabla(t)$ of Eq.~(\ref{YM
heat eq}) with the boundary conditions~(\ref{relative BC})
or~(\ref{absolute BC}) obeys the estimate
\begin{equation}\label{bound ex4}\sup_{x\in
M}\left|R^{\nabla(\rho)}(x)\right|^2_E\le\max\left\{\frac{4\sqrt{\YM(0)}}{\rho^2}\,,\sqrt{\YM(0)}\,\right\},\qquad
\rho\in(0,T),
\end{equation}
if the initial energy $\YM(0)$ is smaller than $\xi$.
\end{theorem}

We turn our attention to dimensions~5 and higher. In this case, the
proof of the result will require the Li-Yau-Hamilton estimate
established in Section~\ref{sec LYH}. This forces us to impose
stronger geometric assumptions on the manifold $M$.

The following theorem yields a bound on $R^{\nabla(\rho)}$ provided
$\YM(0)$ is smaller than a certain value $\xi(\rho)$ depending
on~$\rho\in[0,T)$. This result implies that the curvature of a
solution to Eq.~(\ref{YM heat eq}) cannot blow up after time~$\rho$
if the initial energy does not exceed $\xi(\rho)$. In the above
setting, the connection $\nabla(t)$ is defined for each $t\in[0,T)$
and depends differentiably on $t$ on this interval. Therefore,
$R^{\nabla(t)}$ does not blow up at time $T$ if $\YM(0)<\xi(\rho)$
for some $\rho\in(0,T)$.

\begin{theorem}\label{thm YMH dim5+}
Let the dimension $\dim M$ be greater than or equal to~5. Suppose
the boundary $\partial M$ is totally geodesic. Moreover, suppose
either Assumption~\ref{assu curv bndy} of Theorem~\ref{theorem LYH
no crv} or Assumptions~\ref{assu Ricci par} and~\ref{assu nneg sect}
of Theorem~\ref{theorem LYH crv} are fulfilled for $M$. Then there
exists a positive non-decreasing function $\xi(s)$ on $(0,\infty)$
that depends on nothing but $M$ and satisfies the following
statement: Given $\rho\in(0,T)$, the solution $\nabla(t)$ of
Eq.~(\ref{YM heat eq}) with the boundary conditions~(\ref{relative
BC}) or~(\ref{absolute BC}) obeys the estimate
\begin{equation}\label{bound ex5+}\sup_{x\in
M}\left|R^{\nabla(\rho)}(x)\right|^2_E\le\max\left\{\frac{16\sqrt{\YM(0)}}{\rho^2}\,,\sqrt{\YM(0)}\,\right\}
\end{equation}
if the initial energy $\YM(0)$ is smaller than $\xi(\rho)$.
\end{theorem}

The assertions of Theorems~\ref{thm YMH dim23}, \ref{thm YMH dim4},
and~\ref{thm YMH dim5+} may be refined. We present them here in the
less general form in order to ensure that the technical details do
not obscure the qualitative meaning. The possible refinements are
explained in Remarks~\ref{refine 23}, \ref{refine 4},
and~\ref{refine 5+}.

To prove the three theorems above, we employ the probabilistic
technique developed in~\cite{MARBAT02}. The main stochastic process
to be used for our arguments is a reflecting Brownian motion on the
manifold $M$. Its transition density is the Neumann heat kernel on
$M$. Before introducing the probabilistic machinery, we need to
state two geometric results.

First of all, it is necessary to formulate a version of the
integration by parts formula. Let us recollect some conventions and
notation. The boundary of $M$ carries a natural Riemannian metric
inherited from $M$. The orientation of $\partial M$ is induced by
that of $M$. The integration over~$\partial M$ is to be carried out
with respect to the Riemannian volume measure on~$\partial M$. We
write $\nu$ for the outward unit normal vector field on the
boundary. The letter $\iota$ stands for the interior product.

We are now ready to lay down integration by parts formula. Our
source for this result is the paper~\cite{LGNCpre}.

\begin{lemma}\label{lemma int parts} Let $\nabla$ be a connection in $E$
compatible with the structure group~$G$. Consider $\Ad E$-valued
forms $\phi\in\Omega^p(\Ad E)$ and $\psi\in\Omega^{p+1}(\Ad E)$ with
$p=0,\ldots,\dim M-1$. The equality
\[
\int_M\left(\left<\mathrm{d}_\nabla\phi,\psi\right>_E
-\left<\phi,\mathrm{d}^*_\nabla\psi\right>_E\right)\mathrm{d}
x=\int_{\partial M}\left<\phi,\iota_\nu\psi\right>_E\mathrm{d} x
\] holds true.
\end{lemma}

As mentioned above, an argument involving Lemma~\ref{lemma int
parts} proves that $\YM(t)$ is non-increasing in $t\in[0,T)$; see,
for instance, \cite{JR92,LGNCpre}. This fact is crucial for our
further considerations.

The next step is to understand what Eqs.~(\ref{relative BC})
and~(\ref{absolute BC}) can tell us about the behavior of
$\left|R^{\nabla(t)}(x)\right|_E^2$ near the boundary of $M$. In
order to do this, we present the following result. It may be viewed
as a variant of Lemma~3.1\footnote{This statement was labeled
Lemma~3.1 in a preliminary version of~\cite{LGNCpre}. It may appear
under a different tag in the final manuscript.} in~\cite{LGNCpre}
for manifolds with convex boundary. The proof utilizes a computation
carried out in~\cite{LGNCpre}. Given $\phi\in\Omega^p(\Ad E)$ and
$x\in M$, the notation $\phi(x)$ refers to the restriction of $\phi$
to $(T_xM)^p$.

\begin{lemma}\label{lemma Neumann norm}
Let the boundary $\partial M$ be convex in the sense of~(\ref{def
convex}). Suppose $\nabla$ is a connection in $E$ compatible with
the structure group $G$. Consider an $\Ad E$-valued $p$-form
$\phi\in\Omega^p(\Ad E)$ with $p=0,\ldots,\dim M$. If either the
equations
\begin{equation}\label{rel BC phi}
\phi_{\tan}=0,\qquad\left(\mathrm{d}^*_{\nabla}\phi\right)_{\tan}=0
\end{equation} or the equations
\begin{equation}\label{abs BC phi}
\phi_{\norm}=0,\qquad\left(\mathrm{d}_{\nabla}\phi\right)_{\norm}=0
\end{equation}
are satisfied on $\partial M$, then the formula
\begin{equation}\label{est Neumann norm}
\frac\partial{\partial\nu}\left|\phi(x)\right|^2_E\le0,\qquad
x\in\partial M,
\end{equation}
holds true.
\end{lemma}

\begin{proof}
We begin by selecting a local coordinate system on $M$ convenient
for our arguments. Choose a point $\tilde x\in\partial M$. Let
$\{e_1,\ldots,e_{n-1}\}$ be an orthonormal basis of the space
$T_{\tilde x}\partial M$ such that
\[\II(e_i,e_j)=\delta_i^j\lambda_i,\qquad i,j=1,\ldots,n-1.\]
In this formula, $\delta_i^j$ is the Kronecker symbol, and
$\lambda_i$ are the principal curvatures at $\tilde x$. Since
$\partial M$ is convex, $\lambda_i$ must be nonnegative for all
$i=1,\ldots,n-1$. Take a coordinate neighborhood $U^\partial$ of
$\tilde x$ in $\partial M$ with a coordinate system
$y_1,\ldots,y_{n-1}$ in~$U^\partial$ centered at $\tilde x$. We
assume $\frac\partial{\partial y_i}$ coincides with $e_i$ at $\tilde
x$ for each $i=1,\ldots,n-1$. As in the proof of
Theorem~\ref{theorem LYH no crv}, consider the mapping $\mu(r,x)$
defined on $[0,\epsilon)\times\partial M$ by the formula
$\mu(r,x)=\exp_x(-r\nu)$. The number $\epsilon>0$ is chosen small
enough for $\mu(r,x)$ to be a diffeomorphism onto its image. The set
$U=\mu\left([0,\epsilon)\times U^\partial\right)$ is a neighborhood
of $\tilde x$ in the manifold $M$. We extend $y_1,\ldots,y_{n-1}$ to
a coordinate system $x_1,\ldots,x_n$ in $U$ by demanding that the
equalities
\begin{align*}x_k(\mu(r,x))&=y_k(x),~x_n(\mu(r,x))=r,\\
r&\in[0,\epsilon),~x\in U^\partial,~k=1,\ldots,n-1,
\end{align*}
hold true; cf.~\cite{HM91}. The vector $\frac\partial{\partial x_i}$
coincides with $e_i$ at $\tilde x$ for each $i=1,\ldots,n-1$. It is
easy to see that $\frac\partial{\partial x_i}$ is tangent to the
boundary on the set $U^\partial$ for $i=1,\ldots,n-1$. The vector
field $\frac\partial{\partial x_n}$ coincides with $-\nu$ at every
point of $U^\partial$.

Having fixed a suitable local coordinate system on $M$, we now
proceed to the actual proof of the lemma. Without loss of
generality, suppose Eqs.~(\ref{abs BC phi}) hold for $\phi$ on
$\partial M$. If this is not the case and Eqs.~(\ref{rel BC phi})
hold instead, we can replace $\phi$ with the form $*\,\phi$
satisfying~(\ref{abs BC phi}). (The symbol $*$ denotes the Hodge
star operator.) Since $|\phi(x)|_E$ equals $|{*\,\phi(x)}|_E$ for
all $x\in M$, proving the lemma for~$*\,\phi$ would suffice.

From the technical point of view, it is convenient for us to assume
that $\phi$ belongs to $\Omega^p(\Ad E)$ with $p$ between~1 and
$\dim M$. This restriction is not significant. Indeed, if $\phi$ is
an $\Ad E$-valued 0-form on $M$, then estimate~(\ref{est Neumann
norm}) follows directly from the second formula in~(\ref{abs BC
phi}).

Our next step is to write down an expression for the derivative
$\frac\partial{\partial\nu}\left|\phi(x)\right|^2_E$ using the
coordinate system introduced above. Observe that, in the
neighborhood $U$ of the point $\tilde x$, one can represent $\phi$
by the equality
\[
\phi(x)=\alpha(x)\wedge \mathrm{d} x_n+\beta(x).
\]
Here, $\alpha$ and $\beta$ are $\Ad E$-valued forms defined on $U$
and given by the formulas
\[
\alpha(x)=\sum\alpha_I(x)\,\mathrm{d} x^I,\qquad
\beta(x)=\sum\beta_J(x)\,\mathrm{d} x^J.
\]
The sums are taken over all the multi-indices
$I=(i_1,\ldots,i_{p-1})$ and $J=(j_1,\ldots,j_p)$ with $1\le
i_1<\cdots<i_{p-1}<n$ and $1\le j_1<\cdots<j_p<n$. The mappings
$\alpha_I(x)$ and $\beta_J(x)$ defined on $U$ are local sections of
the bundle $\Ad E$. The notations $\mathrm{d} x^I$ and $\mathrm{d}
x^J$ refer to $\mathrm{d} x_{i_1}\wedge\cdots\wedge \mathrm{d}
x_{i_{p-1}}$ and $\mathrm{d} x_{j_1}\wedge\cdots\wedge \mathrm{d}
x_{j_p}$. If $p=1$, then $\alpha$ should be interpreted as an $\Ad
E$-valued 0-form on $U$. If $p=n$, then $\beta$ equals zero.

Following the computation from~\cite[Proof of Lemma~3.1]{LGNCpre},
we arrive at the formula
\begin{align}\label{aux Nnorm}
\frac12\frac\partial{\partial\nu}|\phi(x)|_E^2&=\sum\left<\beta_J(x),\beta_K(x)\right>_E\left<D_\nu
\mathrm{d} x^J,\mathrm{d} x^K\right>_\Lambda, \nonumber \\ x&\in
U\cap\partial M.
\end{align}
The summation is now carried out over all $J=(j_1,\ldots,j_p)$ and
$K=(k_1,\ldots,k_p)$ with $1\le j_1<\cdots<j_p<n$ and $1\le
k_1<\cdots<k_p<n$. The angular brackets with the lower index
$\Lambda$ stand for the scalar product in $\Lambda T^*M$ induced by
the Riemannian metric on $M$. If $p=n$, then the sum in~(\ref{aux
Nnorm}) should be interpreted as~0.

We have thus laid down an expression for
$\frac\partial{\partial\nu}\left|\phi(x)\right|^2_E$ in our local
coordinates. The next step is to establish estimate~(\ref{est
Neumann norm}) at the point $\tilde x$ using formula~(\ref{aux
Nnorm}). The argument will rely on the properties of the coordinate
system fixed in $U$. Remark that $\tilde x$ was originally chosen as
an arbitrary point in $\partial M$. Therefore,
establishing~(\ref{est Neumann norm}) at this point would suffice to
prove the lemma.

Let us take a closer look at the scalar product $\left<D_\nu
\mathrm{d} x^J,\mathrm{d} x^K\right>_\Lambda$ in the right-hand side
of~(\ref{aux Nnorm}). The formula
\[
\left<D_\nu \mathrm{d} x^J,\mathrm{d}
x^K\right>_\Lambda=\sum_{l=1}^p\det\left(
\begin{array}{ccc}
\bigl<\mathrm{d} x_{j_1},\mathrm{d} x_{k_1}\bigr>_\Lambda &\cdots &
\left<\mathrm{d} x_{j_1},\mathrm{d} x_{k_p}\right>_\Lambda \\ \vdots
& & \vdots \\ \left<\mathrm{d} x_{j_{l-1}},\mathrm{d}
x_{k_1}\right>_\Lambda &\cdots & \left<\mathrm{d}
x_{j_{l-1}},\mathrm{d} x_{k_p}\right>_\Lambda \\
\\ \left<\vphantom{\mathrm{d} x_{j_{l-1}}} D_\nu \mathrm{d}
x_{j_l},\mathrm{d} x_{k_1}\right>_\Lambda &\cdots &
\left<\vphantom{\mathrm{d} x_{j_{l-1}}}D_\nu \mathrm{d}
x_{j_l},\mathrm{d} x_{k_p}\right>_\Lambda \\ \\
 \left<\mathrm{d} x_{j_{l+1}},\mathrm{d} x_{k_1}\right>_\Lambda &\cdots &
\left<\mathrm{d} x_{j_{l+1}},\mathrm{d} x_{k_p}\right>_\Lambda \\ \vdots & & \vdots \\
\left<\mathrm{d} x_{j_p},\mathrm{d} x_{k_1}\right>_\Lambda &\cdots &
\left<\mathrm{d} x_{j_p},\mathrm{d} x_{k_p}\right>_\Lambda
\end{array}\right)
\]
holds on $U\cap\partial M$. Our choice of the coordinate system
provides the identities
\begin{align*}
\left<\mathrm{d} x_l,\mathrm{d} x_m\right>_\Lambda&=\delta_l^m, \\
\langle D_\nu \mathrm{d} x_l,\mathrm{d}
x_m\rangle_\Lambda&=-\II\left(\frac\partial{\partial
x_l},\frac\partial{\partial x_m}\right)=-\delta_l^m\lambda_l, \qquad
l,m=1,\ldots,n-1,
\end{align*}
at the point $\tilde x$. (Recall that $\delta_l^m$ is the Kronecker
symbol, and $\lambda_l$ are the principal curvatures.) As a
consequence,
\[
\left<D_\nu \mathrm{d} x^J,\mathrm{d}
x^K\right>_\Lambda=-\left(\lambda_{j_1}+\cdots+\lambda_{j_p}\right)
\] at $\tilde x$ when $J$ coincides with $K$, and \[
\left<D_\nu \mathrm{d} x^J,\mathrm{d} x^K\right>_\Lambda=0 \] at
$\tilde x$ when $J$ differs from~$K$.

Let us substitute the obtained equalities into~(\ref{aux Nnorm}). We
conclude that
\[
\frac12\frac\partial{\partial\nu}|\phi(x)|_E^2=-\sum\left<\beta_J(x),\beta_J(x)\right>_E\left(\lambda_{j_1}+\cdots+\lambda_{j_p}\right).
\]
at the point $\tilde x$. The summation is carried out over all the
multi-indices $J$ as described above. The scalar product
$\left<\beta_J(x),\beta_J(x)\right>_E$ is greater than or equal to~0
for every $J$. The principal curvatures
$\lambda_{j_1},\ldots,\lambda_{j_p}$ are all nonnegative because
$\partial M$ is convex. As a result, estimate~(\ref{est Neumann
norm}) holds at the point $\tilde x$. This proves the lemma because
$\tilde x$ can be chosen arbitrarily.  \end{proof}

Our intention is to employ the technique developed
in~\cite{MARBAT02} for establishing Theorems~\ref{thm YMH dim23},
\ref{thm YMH dim4}, and~\ref{thm YMH dim5+}. We now introduce the
required probabilistic machinery. Consider the bundle $O(M)$ of
orthonormal frames over $M$. The letter $\pi$ denotes the projection
in this bundle. Let $u_t^Y$ be a horizontal reflecting Brownian
motion on $O(M)$ starting at the frame $Y\in O(M)$. We assume
$u_t^Y$ is defined on the filtered probability space
$(\Omega,\mathcal F,(\mathcal F_t)_{t\in[0,\infty)},\mathbb P)$
satisfying the ``usual hypotheses." The symbol $\mathbb E$ will be
used for the expectation. The rigorous definition of a horizontal
reflecting Brownian motion on the bundle of orthonormal frames can
be found in~\cite[Chapter~V]{NISW89} and in~\cite{EH02a}.

Introduce the process $X^y_t=\pi\left(u^Y_t\right)$. Here, we denote
$y=\pi(Y)$. It is well-known that $X_t^y$ is a reflecting Brownian
motion on $M$ starting at the point $y$. Details can be found
in~\cite[Chapter~V]{NISW89}.

By definition, the process $u^Y_t$ satisfies the equation
\begin{align}\label{Ito u_t}
\mathrm{d} f\left(t,u_t^Y\right)=&\sum_{i=1}^n(\mathcal
H_if)\left(t,u_t^Y\right)\mathrm{d} B_t^i \nonumber \\
&+\left(\frac{\partial}{\partial t}
+\frac12\Delta_{O(M)}\right)f\left(t,u_t^Y\right)\mathrm{d}
t-(\mathcal Nf)\left(t,u_t^Y\right)\mathrm{d} L_t
\end{align}
for every smooth real-valued function $f(t,u)$ on $[0,\infty)\times
O(M)$. Let us describe the objects occurring in the right-hand side.
As before, $n\ge2$ is the dimension of $M$. The notation $\mathcal
H_i$ refers to the canonical horizontal vector fields on~$O(M)$. The
process $(B_t^1,\ldots,B_t^n)$ is an $n$-dimensional Brownian motion
defined on $(\Omega,\mathcal F,(\mathcal
F_t)_{t\in[0,\infty)},\mathbb P)$. The operator $\Delta_{O(M)}$ is
Bochner's horizontal Laplacian. It appears as the sum of $\mathcal
H_i^2$ with $i=1,\ldots,n$. The symbol $\mathcal N$ stands for the
horizontal lift of the vector field $\nu$ on $\partial M$. The
non-decreasing process $L_t$ is the boundary local time. It only
increases when $\pi(u_t^Y)$ belongs to $\partial M$.

Consider a smooth real-valued function $h(t,x)$ on $[0,\infty)\times
M$. Applying~(\ref{Ito u_t}) with $f(t,u)=h(t,\pi(u))$, we obtain an
equation for the process $h(t,X_t^y)$. This simple observation is
important to the proofs of Theorems~\ref{thm YMH dim23}, \ref{thm
YMH dim4}, and~\ref{thm YMH dim5+}. It is also used for establishing
Proposition~\ref{proposition exit time} in the next section. When
$f(t,u)=h(t,\pi(u))$, the formulas
\begin{align}\label{lift of operators}
\Delta_{O(M)}f(t,u)&=\left.\Delta_Mh(t,x)\right|_{x=\pi(u)}, \nonumber \\
(\mathcal Nf)(t,u)
&=\frac{\partial}{\partial\nu}\left.h(t,x)\right|_{x=\pi(u)},\qquad
t\in[0,\infty),~u\in O(M),
\end{align}
hold true.

Let $g(t,x,y)$ denote the transition density of the reflecting
Brownian motion~$X_t^y$. The function $\tilde g_y(t,x)=g(2t,x,y)$ is
a smooth positive solution to the heat equation~(\ref{heat eq p})
with the Neumann boundary condition~(\ref{Neumann BC p}). Note that
the density $g(t,x,y)$ will be playing a significant role in our
further considerations. The estimates required to establish
Theorems~\ref{thm YMH dim23}, \ref{thm YMH dim4}, and~\ref{thm YMH
dim5+} rely on those known for~$g(t,x,y)$.

All the probabilistic objects we will need are now at hand.
Introduce the notation
\[
q(t,x)=\left|R^{\nabla(t)}(x)\right|_E^2,\qquad t\in[0,T),~x\in M.
\]
Given $r\in(0,T)$, define
\[
\zeta^{r,y}(t)=\int_M q\left(r-t,x\right)
g\left(t,x,y\right)\mathrm{d} x,\qquad t\in(0,r].
\]
The quantity $\zeta^{r,y}(t)$ may be interpreted as $\mathbb
E\left(q\left(r-t,X_t^y\right)\right)$. Applying Remark~\ref{rem
p<Ct^n2} to the function~$\tilde g_{y}(t,x)$ and taking the
monotonicity of $\YM(t)$ into account, one concludes that
\begin{equation}\label{g<C_1t^-n/2}
\zeta^{r,y}(t)\le C_1t^{-\frac{\dim M}2}\YM(0),\qquad
t\in(0,\min\{r,1\}],
\end{equation}
with $C_1>0$ determined by~(\ref{p<Ct^-n/2}). We are now in a
position to prove Theorems~\ref{thm YMH dim23} and \ref{thm YMH
dim4}. Two more lemmas are required to consider the case where $\dim
M$ is~5 or higher. We will state them afterwards.

\begin{proof}[Proof of Theorem~\ref{thm YMH dim23}.]
Fix $\rho\in\left(0,T\right)$. Our goal is to obtain a bound on
$\sup_{x\in M}q(\rho,x)$. Choose $\alpha\in(0,1)$ and denote
$\rho_0=\max\left\{0,\rho-\frac1\alpha\right\}$. Let the number
$\sigma_0\in(0,\rho-\rho_0]$ satisfy the equality
\begin{equation}\label{ex23 aux5}
\sigma_0^2\sup_{t\in\left[\rho_0+\sigma_0,\rho\right]}\sup_{x\in
M}q(t,x)=\sup_{\sigma\in[0,\rho-\rho_0]}\left(\sigma^2\sup_{t\in\left[\rho_0+\sigma,\rho\right]}\sup_{x\in
M}q(t,x)\right).
\end{equation}
There exist $t_*\in\left[\rho_0+\sigma_0,\rho\right]$ and $x_*\in M$
such that
\begin{equation}\label{star def}q(t_*,x_*)=\sup_{t\in\left[\rho_0+\sigma_0,\rho\right]}\sup_{x\in M}q(t,x).\end{equation}
It is convenient for us to write $q_0$ instead of $q(t_*,x_*)$. Our
next step is to estimate the number $q_0$. The desired bound on
$\sup_{x\in M}q(\rho,x)$ will follow therefrom.

Using the heat equation~(\ref{YM heat eq}) and the
Bochner-Weitzenb\"{o}ck formula, we can prove the existence of a
constant $C_2>0$ such that
\begin{equation}\label{ex23 aux1}
\left(\frac{\partial}{\partial t}-\frac12\Delta_M\right)q(t,x)\le
C_2\left(1+\sqrt{q(t,x)}\,\right)q(t,x)
\end{equation}
for $t\in[0,T)$ and $x\in M$; see~\cite[Lemma~2.2]{YCCLS94}. The
definition of $\sigma_0$ implies
\begin{align}\label{ex23 aux3}
\sup_{t\in\left[t_*-\alpha\sigma_0,t_*\right]}\sup_{x\in
M}q(t,x)&\le\sup_{t\in\left[\rho_0+(1-\alpha)\sigma_0,\rho\right]}\sup_{x\in
M}q(t,x) \nonumber
\\ &\le\frac{\sigma_0^2}{(1-\alpha)^2\sigma_0^2}\sup_{t\in\left[\rho_0+\sigma_0,\rho\right]}\sup_{x\in
M}q(t,x)=\tilde{\alpha}^2\,q_0
\end{align}
with $\tilde\alpha=\frac1{1-\alpha}$. Inequalities~(\ref{ex23 aux1})
and~(\ref{ex23 aux3}) will play an essential role in estimating the
number~$q_0$. Let $u_t^Y$ be a horizontal reflecting Brownian motion
in the bundle $O(M)$. We suppose $u_t^Y$ starts at a frame $Y$
satisfying $\pi(Y)=x_*$. Define $X^{x_*}_t=\pi\left(u^Y_t\right)$
and consider the process
\[
Z_t=\e^{C_2(1+\tilde\alpha\sqrt{q_0}\,)t}q\left(t_*-t,X_t^{x_*}\right)
\]
for $t\in\left[0,\alpha\sigma_0\right)$. Formulas~(\ref{Ito u_t})
and~(\ref{lift of operators}) yield
\begin{align*}
q_0&=Z_0=\mathbb E(Z_t) \\
&\hphantom{=}~-\mathbb E\left(\int_0^t\left(-\frac\partial{\partial
r}+\frac12\Delta_M\right)
\left.\e^{C_2(1+\tilde\alpha\sqrt{q_0}\,)(t_*-r)}\,q\left(r,X_{t_*-r}^{x_*}\right)\right|_{r=t_*-s}\mathrm{d} s\right)\\
&\hphantom{=}~+\mathbb
E\left(\int_0^t\e^{C_2(1+\tilde\alpha\sqrt{q_0}\,)s}
\frac{\partial}{\partial\nu}\,q\left(t_*-s,X_s^{x_*}\right)\mathrm{d}
L_s\right).
\end{align*}
In view of~(\ref{ex23 aux1}), (\ref{ex23 aux3}), and
Lemma~\ref{lemma Neumann norm}, this implies $q_0\le\mathbb E(Z_t)$
for $t\in\left[0,\alpha\sigma_0\right)$. As a consequence, the
formula
\begin{equation}\label{major est pf23}
q_0\le\e^{C_2(1+\tilde\alpha\sqrt{q_0}\,)t}\zeta^{t_*,x_*}(t),\qquad
t\in\left[0,\alpha\sigma_0\right),
\end{equation}
holds true. We will now use it to prove that
\begin{equation}\label{ex23 aux4}
\sup_{x\in M}q(\rho,x)\le\max\left\{\frac{\YM(0)}{\alpha^2\rho^2}\,,
\theta_1\e^{\theta_{2,\alpha}\sqrt{\YM(0)}}\YM(0)\right\}
\end{equation}
with $\theta_1>0$ and $\theta_{2,\alpha}>0$. Estimate~(\ref{bound
ex23}) will follow by looking at the case where $\alpha=\frac12$.

Let us assume $q_0>0$ and $\YM(0)>0$. This does not lead to any loss
of generality. Indeed, if $q_0=0$, then the supremum $\sup_{x\in
M}q(\rho,x)$ is equal to~0 and~(\ref{ex23 aux4}) holds for any
$\theta_1$ and~$\theta_{2,\alpha}$. When $\YM(0)=0$, we have
$\YM(t_*)=0$ due to the fact that $\YM(t)$ is non-increasing in
$t\in[0,T)$. In this case, $q_0$ equals~0, and~(\ref{ex23 aux4}) is
again satisfied for any $\theta_1$ and~$\theta_{2,\alpha}$.

Denote $t_0=\sqrt{\frac{\YM(0)}{q_0}}$. If $t_0\ge\alpha\sigma_0$,
then
\[
(\rho-\rho_0)^2\sup_{x\in
M}q(\rho,x)\le\sigma_0^2q_0\le\frac{\YM(0)}{\alpha^2}
\]
by virtue of the definitions of $\sigma_0$ and $t_0$. In this case,
the estimate
\begin{align}\label{ex23 aux2}
\sup_{x\in M}q(\rho,x)&\le\frac{\YM(0)}{\alpha^2(\rho-\rho_0)^2}
\nonumber \\
&=\frac{\YM(0)}{\alpha^2\left(\min\left\{\rho,\frac1\alpha\right\}\right)^2}=\max
\left\{\frac{\YM(0)}{\alpha^2\rho^2},\YM(0)\right\}
\end{align}
holds true, which means~(\ref{ex23 aux4}) is satisfied for all
$\theta_1\ge1$ and $\theta_{2,\alpha}>0$. If $t_0<\alpha\sigma_0$
(note that $\alpha\sigma_0\le\alpha(\rho-\rho_0)\le1$), then
formulas~(\ref{major est pf23}) and~(\ref{g<C_1t^-n/2}) yield
\[
q_0\le\e^{C_2(1+\tilde\alpha\sqrt{q_0}\,)t_0}\zeta^{t_*,x_*}\left(t_0\right)\le\e^{C_2\tilde\alpha\sqrt{\YM(0)}}\tilde
Cq_0^{\frac{\dim M}4}\YM(0)^{\frac{4-\dim M}4}
\]
with $\tilde C=\e^{C_2}C_1$. Hence
\begin{align*}
\sup_{x\in M}q(\rho,x)&\le
q_0\le\left(\e^{C_2\tilde\alpha\sqrt{\YM(0)}}\tilde
C\YM(0)^{\frac{4-\dim
M}4}\right)^{\frac4{4-\dim M}} \\
&=\left(\e^{C_2\tilde\alpha\sqrt{\YM(0)}}\tilde
C\right)^{\frac4{4-\dim M}}\YM(0).
\end{align*}
Combined with~(\ref{ex23 aux2}), this estimate shows that~(\ref{ex23
aux4}) holds for
\[ \theta_1=\max\left\{\tilde C^{\frac4{4-\dim
M}},1\,\right\},\qquad
 \theta_{2,\alpha}=\frac4{4-\dim M}\,C_2\tilde\alpha.
\]
We now assume $\alpha=\frac12$. The desired result follows at once.
The role of the constant $\theta_2$ is to be played by
$\theta_{2,\frac12}$.  \end{proof}

\begin{remark}\label{refine 23}
While proving the theorem, we have actually established a stronger
result. Namely, take a number $\alpha$ from the interval $(0,1)$.
Suppose the conditions of Theorem~\ref{thm YMH dim23} are satisfied.
Then the estimate
\[
\sup_{x\in
M}\left|R^{\nabla(\rho)}(x)\right|^2_E\le\max\left\{\frac{\YM(0)}{\alpha^2\rho^2}\,,
\theta_1\e^{\theta_{2,\alpha}\sqrt{\YM(0)}}\YM(0)\right\},\qquad\rho\in(0,T),
\]
holds true. In the right-hand side, $\theta_1>0$ is a constant
depending only on $M$, whereas $\theta_{2,\alpha}>0$ is determined
by $\alpha$ and $M$. When formulating Theorem~\ref{thm YMH dim23},
we restricted our attention to the case where $\alpha=\frac12$. This
was done for the sake of simplicity and understandability.
\end{remark}

\begin{proof}[Proof of Theorem~\ref{thm YMH dim4}.]
Fix $\rho\in\left(0,T\right)$, $\alpha\in(0,1)$, and
$\beta\in(0,1)$. Denote
$\rho_0=\max\left\{0,\rho-\frac1\alpha\right\}$. Let
$\sigma_0\in(0,\rho-\rho_0]$,
$t_*\in\left[\rho_0+\sigma_0,\rho\right]$, and $x_*\in M$ satisfy
Eqs.~(\ref{ex23 aux5}) and~(\ref{star def}). We write $q_0$ instead
of $q(t_*,x_*)$. Our next step is to demonstrate that
\begin{equation}\label{ex4 aux1}
\sigma_0^2q_0\le\frac{\YM(0)^\beta}{\alpha^2}
\end{equation} provided $\YM(0)$ is smaller than a number $\xi_{\alpha,\beta}>0$
depending only on $\alpha$,~$\beta$, and the manifold $M$. The
assertion of the theorem will be deduced from this estimate.

Suppose $\YM(0)=0$. Then $\YM(t_*)=0$ due to the monotonicity
of~$\YM(t)$ in $t\in[0,T)$. Ergo, $q_0$ is equal to~0. It becomes
evident that $\sigma_0^2q_0=\frac{\YM(0)^\beta}{\alpha^2}$.

We have thus proved~(\ref{ex4 aux1}) in the case where $\YM(0)=0$.
Let us consider the general situation. Assume~(\ref{ex4 aux1}) fails
to hold. Then $q_0>0$, $\YM(0)>0$, and the number
$t'=\sqrt{\frac{\YM(0)^\beta}{q_0}}$ lies in the interval
$(0,\alpha\sigma_0)\subset(0,1)$. Repeating the arguments from the
proof of Theorem~\ref{thm YMH dim23} and using~(\ref{g<C_1t^-n/2}),
we conclude that the inequality
\[
q_0\le\e^{C_2(1+\tilde\alpha\sqrt{q_0}\,)t'}\,\zeta^{t_*,x_*}\left(t'\right)
\le\e^{C_2\tilde\alpha\sqrt{\YM(0)^\beta}}\tilde
Cq_0\YM(0)^{1-\beta}
\]
must be satisfied. Here, $\tilde\alpha$ stands for
$\frac1{1-\alpha}$. The constant $\tilde C$ appears as
$\e^{C_2}C_1$. It is easy to see, however, that the above inequality
fails when
\[\YM(0)<\xi_{\alpha,\beta}=\min\left\{\left(\e^{C_2\tilde\alpha}\tilde C\right)^{-\frac1{1-\beta}},1\right\}.\] This contradiction
establishes~(\ref{ex4 aux1}) under the condition
$\YM(0)<\xi_{\alpha,\beta}$.

In order to complete the proof of the theorem, we estimate
$\sup_{x\in M}q(\rho,x)$. The definition of $\sigma_0$ suggests that
\[
(\rho-\rho_0)^2\sup_{x\in M}q(\rho,x)\le\sigma_0^2q_0.\] In view
of~(\ref{ex4 aux1}), this implies
\begin{align*}
\sup_{x\in
M}q(\rho,x)&\le\frac{\YM(0)^\beta}{\alpha^2(\rho-\rho_0)^2}
\\ &=\frac{\YM(0)^\beta}{\alpha^2\left(\min\left\{\rho,\frac1\alpha\right\}\right)^2}=\max\left\{
\frac{\YM(0)^\beta}{\alpha^2\rho^2}\,,\YM(0)^\beta\right\}
\end{align*}
provided $\YM(0)<\xi_{\alpha,\beta}$. The assertion of the theorem
follows by assuming $\alpha=\beta=\frac12$. Inequality~(\ref{bound
ex4}) holds when $\YM(0)<\xi=\xi_{\frac12,\frac12}$.
\end{proof}

\begin{remark}\label{refine 4}
In the course of the proof, we have actually established a result
stronger than Theorem~\ref{thm YMH dim4}. Namely, fix
$\alpha\in(0,1)$ and $\beta\in(0,1)$. Suppose the conditions of
Theorem~\ref{thm YMH dim4} are satisfied. If $\YM(0)$ is smaller
than $\xi_{\alpha,\beta}$, then the estimate
\[\sup_{x\in
M}\left|R^{\nabla(\rho)}(x)\right|^2_E
\le\max\left\{\frac{\YM(0)^\beta}{\alpha^2\rho^2}\,,\YM(0)^\beta\right\},\qquad\rho\in(0,T),
\]
holds true. Here, $\xi_{\alpha,\beta}$ is a number depending on
$\alpha$, $\beta$, and $M$. When formulating Theorem~\ref{thm YMH
dim4}, we restricted our attention to $\alpha=\beta=\frac12$. This
was done in order to make the statement more understandable.
\end{remark}

Let us concentrate on the case where $\dim M$ is~5 or higher. First
of all, we need a few auxiliary identities. Their purpose is to help
us obtain a monotonicity formula related to the Yang-Mills heat
equation~(\ref{YM heat eq}). We establish these identities in
Lemma~\ref{lemma formulas for monot} below. The proof is quite
transparent yet worthy of attention. It demonstrates vividly how the
boundary conditions imposed on $R^{\nabla(t)}$ interact with those
satisfied by $g(t,x,y)$. In a way, this interplay of boundary
conditions explains why the Brownian motion used to implement the
probabilistic technique in our context should be reflected at
$\partial M$.

Desiring to remain at the higher level of abstraction, we state
Lemma~\ref{lemma formulas for monot} for a generic $\Ad E$-valued
form $\phi$ and a generic function $f(x)$ on $M$. In our further
arguments, it will be applied with $\phi$ equal to the curvature
$R^{\nabla(r-t)}$ and $f(x)$ equal to the density $g(t,x,y)$.

\begin{lemma}\label{lemma formulas for monot} Let $\nabla$ be a
connection in $E$ compatible with the structure group~$G$. Suppose
$f(x)$ is a real-valued function on $M$ such that
$\frac\partial{\partial\nu}f(x)=0$ on $\partial M$. Consider an $\Ad
E$-valued $p$-form $\phi\in\Omega^p(\Ad E)$ with $p=1,\ldots,\dim
M$. If either Eqs.~(\ref{rel BC phi}) or Eqs.~(\ref{abs BC phi}) are
satisfied for $\phi$ on $\partial M$, then the following formulas
hold true:
\begin{align}\label{int monot}
\int_M\left|\phi\right|_E \Delta_Mf\,\mathrm{d} x&=-\int_M\left<
\grad\left|\phi\right|_E,\grad f\right>\mathrm{d} x, \nonumber
\\
\int_M\left<\mathrm{d}_\nabla \mathrm{d}_\nabla^*\phi,f\phi\right>_E
\mathrm{d}
x&=\int_M\left<\mathrm{d}_\nabla^*\phi,\mathrm{d}_\nabla^*(f\phi)\right>_E\mathrm{d}
x,
\nonumber \\
\int_M\left<\mathrm{d}_\nabla\left(\iota_{\grad\log
f}\phi\right),f\phi\right>_E\mathrm{d} x
&=\int_M\left<\iota_{\grad\log
f}\phi,\mathrm{d}_\nabla^*\left(f\phi\right)\right>_E\mathrm{d} x.
\end{align}
\end{lemma}

\begin{proof}
The first identity in~(\ref{int monot}) is a direct consequence of
the Stokes theorem and the fact that
$\frac\partial{\partial\nu}f(x)=0$. The second one can be deduced
from Lemma~\ref{lemma int parts} in a straightforward fashion.
Notably, the same argument has to be used when proving $\YM(t)$ is
non-increasing in $t\in[0,T)$; see~\cite{LGNCpre}. We will now
establish the third identity in~(\ref{int monot}).

Let us assume Eqs.~(\ref{rel BC phi}) are satisfied for~$\phi$. The
case where Eqs.~(\ref{abs BC phi}) are satisfied instead can be
treated similarly. We will show that the scalar product
$\left<\iota_{\grad\log f}\phi,\iota_\nu\left(f\phi\right)\right>_E$
vanishes on $\partial M$. In view of Lemma~\ref{lemma int parts},
the third identity in~(\ref{int monot}) would follow from this fact
as an immediate consequence.

Observe that the formula $\frac\partial{\partial\nu}f(x)=0$ implies
$\frac\partial{\partial\nu}\log f(x)=0$. Accordingly, the gradient
$\grad\log f$ is tangent to $\partial M$ at every point of~$\partial
M$. This allows us to assume $\phi$ belongs to $\Omega^p(\Ad E)$
with $p$ between~2 and $\dim M$. Indeed, if $\phi$ is an $\Ad
E$-valued 1-form on $M$, then $\iota_{\grad\log f}\phi=0$ due to the
first formula in~(\ref{rel BC phi}).

Take a point $\tilde x\in\partial M$. Choose an orthonormal basis
$\{e_1,\ldots,e_n\}$ of the tangent space $T_{\tilde x}M$ demanding
that $e_n$ coincide with $\nu$. The equality
\begin{align*}
\langle\iota_{\grad\log f}&\phi,\iota_\nu(f\phi)\rangle_E \\
&=\sum\left<\phi\left(\grad\log f,e_{i_1},\ldots,e_{i_{p-1}}\right),
f\phi\left(\nu,e_{i_1},\ldots,e_{i_{p-1}}\right)\right>_E
\end{align*}
holds at $\tilde x$. The summation is to be carried out over all the
arrays $(i_1,\ldots,i_{p-1})$ with $1\le i_1<\cdots<i_{p-1}\le n$.
It is easy to see that
$f\phi\left(\nu,e_{i_1},\ldots,e_{i_{p-1}}\right)$ vanishes when
$i_{p-1}=n$. At the same time, $\phi\left(\grad\log
f,e_{i_1},\ldots,e_{i_{p-1}}\right)$ vanishes when $i_{p-1}<n$
because $\grad\log f$ is tangent to $\partial M$ and
$\phi_{\tan}=0$. We conclude that the scalar product
$\left<\iota_{\grad\log f}\phi,\iota_\nu\left(f\phi\right)\right>_E$
equals~0 at $\tilde x$. Hence the third identity in~(\ref{int
monot}).  \end{proof}

The following lemma states a monotonicity formula related to the
Yang-Mills heat equation~(\ref{YM heat eq}). It is an important step
in establishing Theorem~\ref{thm YMH dim5+} by means of the
probabilistic technique. We emphasize that the proof of the lemma
requires the Li-Yau-Hamilton estimate obtained in Section~\ref{sec
LYH}. For relevant results, see~\cite{MARBAT02} and
also~\cite{RH93b,YCCLS94}.

\begin{lemma}\label{lemma monot}
Let the boundary $\partial M$ be totally geodesic. Suppose either
Assumption~\ref{assu curv bndy} of Theorem~\ref{theorem LYH no crv}
or Assumptions~\ref{assu Ricci par} and~\ref{assu nneg sect} of
Theorem~\ref{theorem LYH crv} are fulfilled for~$M$. Given
$r\in(0,T)$ along with $y\in M$, the formula
\[
\zeta^{r,y}(t_1)\le
\frac1{t_1^2}\left(t_2^2\e^{u(t_2)}\zeta^{r,y}(t_2)+C_3(t_2-t_1)\YM(0)\right)
\]
holds for all $t_1,t_2\in(0,\min\{r,1\})$ satisfying $t_1<t_2$.
Here, $u(t)$ is a positive increasing function on $(0,1]$ such that
$\lim_{t\to0}u(t)=0$, and $C_3>0$ is a constant. Both $u(t)$ and
$C_3$ are determined solely by the manifold $M$.
\end{lemma}
\begin{proof}
First, suppose Assumption~\ref{assu curv bndy} of
Theorem~\ref{theorem LYH no crv} is satisfied. We will consider the
other case later. Proposition~3.4 and Theorem~3.7 in~\cite{MARBAT02}
prove the assertion of the lemma on closed manifolds. The same line
of reasoning works in our situation. However, two points need to be
clarified:
\begin{itemize}
\item
The equality between expressions~(3.10) and~(3.11)
of~\cite{MARBAT02} holds in our setting due to Lemma~\ref{lemma
formulas for monot}. The same can be said about expressions (3.14)
and~(3.15) of that paper.
\item
In order to obtain estimate~(3.22) of~\cite{MARBAT02} for the
Neumann heat kernel $g(t,x,y)$, one should apply formula~(\ref{LYH
no curv}) above to the function $\tilde g_{y}(t,x)=g(2t,x,y)$.
\end{itemize}
The other arguments from the proofs of Proposition~3.4 and
Theorem~3.7 in~\cite{MARBAT02} work in our situation without
significant modifications.

We now consider the case when there are curvature restrictions
imposed on $M$ away from the boundary. More specifically, suppose
Assumptions~\ref{assu Ricci par} and~\ref{assu nneg sect} of
Theorem~\ref{theorem LYH crv} are satisfied. Then the assertion of
the lemma can be established by repeating the arguments from the
proofs of Proposition~3.4 and Theorem~3.7 in~\cite{MARBAT02}. The
required estimate on $g(t,x,y)$ comes from formula~(\ref{LYH crv})
in the present paper applied to the function $\tilde g_y(t,x)$.
 \end{proof}

We are now ready to prove Theorem~\ref{thm YMH dim5+}. Afterwards,
three important remarks will be made.

\begin{proof}[Proof of Theorem~\ref{thm YMH dim5+}.]
Fix $\rho\in\left(0,T\right)$, $\alpha\in(0,1)$, and
$\beta\in(0,1)$. We denote
$\rho_0=\max\left\{(1-\alpha)\rho,\rho-\frac1\alpha\right\}$. Let
$\sigma_0\in(0,\rho-\rho_0]$,
$t_*\in\left[\rho_0+\sigma_0,\rho\right]$, and $x_*\in M$ obey
Eqs.~(\ref{ex23 aux5}) and~(\ref{star def}). Set $q_0=q(t_*,x_*)$.
We will show that
\begin{equation}\label{ex5+ aux2}\sigma_0^2q_0\le\frac{\YM(0)^\beta}{\alpha^2}
\end{equation}
provided $\YM(0)$ is smaller than a certain value
$\xi_{\alpha,\beta}(\rho)$ depending on $\rho$ as a non-decreasing
function. The assertion of the theorem will be deduced from this
estimate. Note that, aside from $\rho$, the value
$\xi_{\alpha,\beta}(\rho)$ only depends on $\alpha$, $\beta$, and
the manifold $M$.

Suppose $\YM(0)=0$. Then $\YM(t_*)=0$ due to the monotonicity of
$\YM(t)$ in $t\in[0,T)$. As a consequence, $q_0$ is equal to~0. We
conclude that~(\ref{ex5+ aux2}) is satisfied when $\YM(0)=0$.

Denote $T_0=\min\left\{\rho_0+\alpha\sigma_0,1\right\}$. Observe
that $\alpha\sigma_0\le T_0<t_*$. This fact is essential because it
will allow us to apply Lemma~\ref{lemma monot} further in the proof.
Assume estimate~(\ref{ex5+ aux2}) fails to hold. Then $q_0>0$,
$\YM(0)>0$, and the number $t'=\sqrt{\frac{\YM(0)^\beta}{q_0}}$ lies
in the interval $(0,\alpha\sigma_0)\subset(0,1)$. The arguments from
the proof of Theorem~\ref{thm YMH dim23} yield
\[
q_0\le\e^{C_2\left(1+\tilde\alpha\sqrt{q_0}\,\right)t'}\zeta^{t_*,x_*}\left(t'\right)
\le\e^{C_2\left(1+\tilde\alpha\sqrt{\YM(0)^\beta}\,\right)}\zeta^{t_*,x_*}\left(t'\right).
\]
Here, the number $\tilde\alpha$ equals $\frac1{1-\alpha}$.
Lemma~\ref{lemma monot} implies
\[
\zeta^{t_*,x_*}\left(t'\right)\le
C'\frac{q_0}{\YM(0)^\beta}\left(T_0^2\zeta^{t_*,x_*}(T_0)+T_0\YM(0)\right)
\]
with $C'=\max\left\{\e^{u(1)},C_3\right\}$. (Note that
Theorems~\ref{theorem LYH no crv} and~\ref{theorem LYH crv} are
being used at this point. More precisely, the proof of
Lemma~\ref{lemma monot} relies on them.) Formula~(\ref{g<C_1t^-n/2})
and the definition of $T_0$ enable us to conclude that
\begin{align*}
q_0&\le\e^{C_2\left(1+\tilde\alpha\sqrt{\YM(0)^\beta}\,\right)}C'\frac{q_0}{\YM(0)^\beta}
\left(T_0^2\zeta^{t_*,x_*}(T_0)+T_0\YM(0)\right)
\\ &\le \e^{C_2\tilde\alpha\sqrt{\YM(0)^\beta}}C''q_0\left(C_1T_0^{2-\frac{\dim
M}2}\YM(0)^{1-\beta}+\YM(0)^{1-\beta}\right)
\\ &\le\e^{C_2\tilde\alpha\sqrt{\YM(0)^\beta}}C''q_0\YM(0)^{1-\beta}\left(C_1
(\min\{(1-\alpha)\rho,1\})^{2-\frac{\dim M}2}+1\right)
\end{align*}
with $C''=\e^{C_2}C'$. However, this is impossible when
\begin{align*}\YM(0)<\xi_{\alpha,\beta}(\rho)&=\min\left\{\xi_{\alpha,\beta}^1(\rho),\xi_{\alpha,\beta}^2(\rho),1\right\}, \\
\xi_{\alpha,\beta}^1(\rho)&=\left(2\e^{C_2\tilde\alpha}C''C_1{(\min\{(1-\alpha)\rho,1\})^{2-\frac{\dim
M}{2}}}\right)^{-\frac1{1-\beta}},
\\ \xi_{\alpha,\beta}^2(\rho)&=\left(2\e^{C_2\tilde\alpha}C''\right)^{-\frac1{1-\beta}}.
\end{align*} The present contradiction establishes~(\ref{ex5+
aux2}) under the condition $\YM(0)<\xi_{\alpha,\beta}(\rho)$.

To complete the proof of the theorem, we need to estimate
$\sup_{x\in M}q(\rho,x)$. The definition of $\sigma_0$ suggests that
\[
(\rho-\rho_0)^2\sup_{x\in M}q(\rho,x)\le\sigma_0^2q_0.\] According
to formula~(\ref{ex5+ aux2}), this implies
\begin{align*}
\sup_{x\in
M}q(\rho,x)&\le\frac{\YM(0)^\beta}{\alpha^2(\rho-\rho_0)^2}
\\ &=\frac{\YM(0)^\beta}{\alpha^2\left(\min\left\{\alpha\rho,\frac1\alpha\right\}\right)^2}
=\max\left\{
\frac{\YM(0)^\beta}{\alpha^4\rho^2}\,,\YM(0)^\beta\right\}
\end{align*}
provided $\YM(0)<\xi_{\alpha,\beta}(\rho)$. We now assume
$\alpha=\beta=\frac12$. The assertion of the theorem follows at
once. Inequality~(\ref{bound ex5+}) holds when
$\YM(0)<\xi(\rho)=\xi_{\frac12,\frac12}(\rho)$.  \end{proof}

\begin{remark}\label{refine 5+}
While proving the theorem, we have really established a stronger
result. That is, suppose $\alpha\in(0,1)$ and $\beta\in(0,1)$. Let
the conditions of Theorem~\ref{thm YMH dim5+} be fulfilled. Given
$\rho\in(0,T)$, if $\YM(0)$ is smaller than
$\xi_{\alpha,\beta}(\rho)$, then the estimate
\[\sup_{x\in
M}\left|R^{\nabla(\rho)}(x)\right|^2_E
\le\max\left\{\frac{\YM(0)^\beta}{\alpha^4\rho^2}\,,\YM(0)^\beta\right\}
\]
is satisfied. Here, $\xi_{\alpha,\beta}(s)$ is a positive
non-decreasing function on $(0,\infty)$ entirely determined by
$\alpha$, $\beta$, and $M$. In the formulation of Theorem~\ref{thm
YMH dim5+}, we only dealt with the case where
$\alpha=\beta=\frac12$. This specific framework was meant to make
the statement more understandable.
\end{remark}

\begin{remark}\label{rem BCR vs BCdel}
In the beginning of Section~\ref{sec YMH}, we imposed the boundary
conditions~(\ref{relative BC}) or~(\ref{absolute BC}) on the
curvature form $R^{\nabla(t)}$. Another approach is feasible.
Namely, one may formulate the boundary conditions for the connection
$\nabla(t)$ directly. The paper~\cite{LGNCpre} takes this particular
standpoint; see also~\cite{AM92,WEG06}. It may or may not be more
natural to impose the boundary conditions on $\nabla(t)$ than to
impose ones on $R^{\nabla(t)}$ depending on the considered problem
and the chosen perspective. However, the approach adopted in the
present paper seems to be technically simpler. The reason for this
lies in the fact that, unlike $\nabla(t)$, the curvature form
$R^{\nabla(t)}$ transforms as a tensor under changes of coordinates.
In particular, it is meaningful to talk about the tangential and the
normal components of $R^{\nabla(t)}$.
\end{remark}

\begin{remark}\label{rem BC relation}
In several situations, imposing the boundary conditions on the
connection is virtually equivalent to imposing ones on its curvature
form. Let us present an example. If a time-dependent connection
satisfies the heat equation~(\ref{YM heat eq}) and the conductor
boundary condition in the sense of~\cite{LGNCpre}, then
formulas~(\ref{relative BC}) can be proved for its curvature. The
converse statement holds with an adjustment. Roughly speaking, the
first formula in~(\ref{relative BC}) ensures that $\nabla(t)$ can be
gauge transformed locally into a connection satisfying the conductor
boundary condition. We refer to~\cite{LGNCpre} for further details.
\end{remark}

\section{An exit time estimate on manifolds with convex boundary}\label{sec exit time}

Let $u_t^Y$ be a horizontal reflecting Brownian motion on the bundle
$O(M)$. It is assumed that this process starts at the frame $Y\in
O(M)$. We consider $u_t^Y$ on a filtered probability space
$(\Omega,\mathcal F,(\mathcal F_t)_{t\in[0,\infty)},\mathbb P)$
satisfying the ``usual hypotheses." The definition and the basic
properties of a horizontal reflecting Brownian motion on $O(M)$ were
discussed in Section~\ref{sec YMH}.

As before, we set $X^y_t=\pi\left(u^Y_t\right)$ with $y=\pi(Y)$. The
process $X_t^y$ is a reflecting Brownian motion on $M$ starting at
the point $y$. This section offers an exit time estimate for $X^y_t$
under the assumption that $\partial M$ is convex. Basically, we
obtain an analogue of Lemma~4.1 in~\cite{MARBAT02};
cf.~\cite[Theorem~3.6.1]{EH02b}. This result enables us to prove
another estimate for the curvature of the solution $\nabla(t)$ to
the Yang-Mills heat equation~(\ref{YM heat eq}). More precisely, we
will establish an analogue of Theorem~4.2 in~\cite{MARBAT02} for
manifolds with boundary.

Additional notation should be introduced at this stage. Let
$\dist_y(x)$ stand for the distance between $y$ and $x\in M$ with
respect to the Riemannian metric on~$M$. Given a radius $r>0$,
consider the ball $B(y,r)=\{x\in M\,|\,\dist_y(x)<r\}$. Its closure
will be denoted by $\bar B(y,r)$. Define
\[\tau(y,r)=\inf\left\{t\ge0\,\left|\,X_t^y\in M\setminus\bar B(y,r)\right.\right\}.\]
In other words, $\tau(y,r)$ is the first exit time of the reflecting
Brownian motion $X_t^y$ from $\bar B(y,r)$. We now lay down an
estimate for $\tau(y,r)$.

\begin{proposition}\label{proposition exit time}
Let the boundary $\partial M$ be convex in the sense of~(\ref{def
convex}). There exist constants $\kappa_0>0$ and $\eta>0$ depending
only on $M$ such that the estimate
\[
\mathbb P\left\{\tau(y,r)<\kappa r^2\right\}\le\e^{-\frac\eta\kappa}
\]
holds for every $r\in(0,1)$ and $\kappa\in(0,\kappa_0)$.
\end{proposition}

The proof of the proposition will be based on Lemma~\ref{lemma dist}
below and on Bernstein's inequality related to local martingales. We
should emphasize that $\kappa_0$ and $\eta$ are fully determined by
the manifold~$M$. In particular, they do not depend on the starting
point $y$ of the reflecting Brownian motion $X_t^y$.

Consider the squared distance function
$\dist^2_y(x)=\left(\dist_y(x)\right)^2$ on~$M$. The following lemma
discusses the analytical features of $\dist_y^2(x)$. Generally
speaking, the behavior of the squared distance function on a
manifold with boundary is quite complicated. The
dissertation~\cite{FEW85} offers a series of results on the subject
and a detailed description of relevant literature. The
paper~\cite{SADBRB93} is a slightly more recent reference. In our
particular situation, however, $\dist_y^2(x)$ behaves nicely because
$\partial M$ is assumed to be convex. Many properties of
$\dist_y^2(x)$ resemble those of the squared distance function on a
closed manifold.

\begin{lemma}\label{lemma dist}
Let the boundary $\partial M$ be convex in the sense of~(\ref{def
convex}). There exists a constant $\epsilon_M>0$ independent of $y$
such that the following statements are satisfied:
\begin{enumerate}
\item\label{lem dist st1}
The squared distance function $\dist_{y}^2(x)$ is smooth in $x$ on
the ball~$B(y,\epsilon_M)$.
\item\label{lem dist st2}
The normal derivative $\frac\partial{\partial\nu}\dist^2_{y}(x)$ is
nonnegative for all $x\in B(y,\epsilon_M)\cap\partial M$.
\item\label{lem dist st3}
There is a constant $K_{\epsilon_M}>0$ independent of $y$ such that
$\Delta_M\dist^2_{y}(x)$ is less than or equal to $K_{\epsilon_M}$
for all $x\in B(y,\epsilon_M)$.
\end{enumerate}
\end{lemma}

\begin{proof}
Our reasoning will be based on embedding $M$ isometrically into a
smooth connected Riemannian manifold $N$ without boundary. Let us
introduce some notation and state a definition. We write
$\dist_{N;z'}(z)$ for the distance from $z'\in N$ to $z\in N$ with
respect to the Riemannian metric on $N$. Accordingly, $B_N(z',r)$
denotes the ball $\{z\in N\,|\,\dist_{N;z'}(z)<r\}$ in $N$ of radius
$r>0$. A set $Q\subset N$ is said to be strongly convex in $N$ if
every pair of distinct points from $Q$ can be joined by a unique (up
to parametrization) minimizing geodesic segment that lies in $Q$.

The first step of the proof is to specify the constant $\epsilon_M$.
Afterwards, we will establish Statements~\ref{lem dist st1},
\ref{lem dist st2}, and~\ref{lem dist st3} for this constant. There
exist a smooth connected Riemannian manifold $N$ without boundary
and a mapping $e$ from $M$ to $N$ such that the following
requirements are satisfied:
\begin{itemize}
\item
The dimension $\dim N$ is equal to $\dim M$.
\item
The mapping $e$ is an isometric embedding.
\item
For every $z\in e(M)$, there is a number $\epsilon_1(z)>0$ such that
the set $B_N(z,\epsilon_1(z))\cap e(M)$ is strongly convex in $N$.
\end{itemize}
The existence of $N$ and $e$ is a standard consequence of~(\ref{def
convex}); see, for instance,~\cite{SK79} and also~\cite{MS01}. In
identifying~$\epsilon_M$, it will be more convenient for us to work
with the image $e(M)$ than with $M$ itself.

We need to state one basic fact about the manifold $N$. Namely,
given a point $z'\in N$, there is a number $\epsilon_2(z')>0$ such
that the following requirement is fulfilled: For every $z''\in
B_N(z',\epsilon_2(z'))$, the inverse exponential map
$\exp^{-1}_{z''}$ is a diffeomorphism from $B_N(z',\epsilon_2(z'))$
onto $\exp^{-1}_{z''}(B_N(z',\epsilon_2(z')))$. It is easy to see
that $\dist_{N;z''}(z)$ is smooth in $z$ on the set
$B_N(z',\epsilon_2(z'))\setminus\{z''\}$, and so is the function
$\dist^2_{N;z''}(z)=\left(\dist_{N;z''}(z)\right)^2$ on the ball
$B_N(z',\epsilon_2(z'))$. This concludes the preparations we had to
make before identifying $\epsilon_M$.

Set $\epsilon_M(z)$ equal to $\frac12\min\{\epsilon_1(z),
\epsilon_2(z)\}$ for $z\in e(M)$. The open cover
$\left(B_N(z,\epsilon_M(z))\right)_{z\in e(M)}$ of the compact set
$e(M)$ has a finite subcover
$\left(B_N(z_i,\epsilon_M(z_i))\right)_{i=1,\ldots,m}$. Define
\begin{equation}\label{epsilon M}\epsilon_M=
\min\left\{\epsilon_M(z_1),\ldots,\epsilon_M(z_m)\right\}.\end{equation}
For every point $z\in e(M)$, there is an index $i$ between~1 and $m$
such that the formulas
\begin{align}\label{dist inclusion} B_N(z,\epsilon_M)&\subset
B_N(z_i,\epsilon_1(z_i)), \nonumber \\ B_N(z,\epsilon_M)&\subset
B_N(z_i,\epsilon_2(z_i))
\end{align}
hold true. This obvious fact plays an important role in our further
arguments.

We will now prove Statement~\ref{lem dist st1} of the lemma for the
constant $\epsilon_M$ specified by~(\ref{epsilon M}). Using the
first inclusion in~(\ref{dist inclusion}), one can show that the
image $e(B(y,\epsilon_M))$ coincides with $B_N(e(y),\epsilon_M)\cap
e(M)$. Furthermore, the equality
\begin{equation}\label{dist eq}
\dist_{y}(x)=\dist_{N;e(y)}(e(x)),\qquad x\in
B(y,\epsilon_M),\end{equation} is satisfied. The second inclusion
in~(\ref{dist inclusion}) implies that the function
$\dist^2_{N;e(y)}(z)$ is smooth on the ball $B_N(e(y),\epsilon_M)$.
The embedding $e$ is a diffeomorphism onto its image. As a result,
$\dist^2_{y}(x)$ must be smooth on $e^{-1}(B_N(e(y),\epsilon_M))$.
Since $e(B(y,\epsilon_M))$ coincides with $B_N(e(y),\epsilon_M)\cap
e(M)$, the preimage $e^{-1}\left(B_N(e(y),\epsilon_M)\right)$ is
equal to $B(y,\epsilon_M)$. Hence the desired smoothness
of~$\dist^2_{y}(x)$.

Let us establish Statement~\ref{lem dist st2} of the lemma. In order
to prove the nonnegativity of
$\frac\partial{\partial\nu}\dist^2_{y}(x)$, we need to compute the
gradient $\grad\dist^2_{y}(x)$. Formula~(\ref{dist eq}) yields
\begin{equation}\label{grad dist sq}\grad\dist^2_{y}(x)=
(\mathrm{d} e)^{-1}
\left.\left(\grad\dist_{N;e(y)}^2(z)\right)\right|_{z=e(x)}\end{equation}
when $x\in B(y,\epsilon_M)$. Our next step is to identify
$\grad\dist_{N;e(y)}^2(z)$ on the image of the ball
$B(y,\epsilon_M)$ under the embedding~$e$.

As stated above, $e(B(y,\epsilon_M))$ is equal to
$B_N(e(y),\epsilon_M)\cap e(M)$. By virtue of the first inclusion
in~(\ref{dist inclusion}), this fact implies the existence of an
index $i$ between~1 and $m$ such that $e(B(y,\epsilon_M))$ is
contained in $B_N(z_i,\epsilon_1(z_i))\cap e(M)$. The latter set is
strongly convex in $N$. In consequence, the following property must
hold: For every $z\in e(B(y,\epsilon_M))\setminus e(\{y\})$, there
is a minimizing geodesic segment $\gamma_z(s)$ that starts at~$z$,
ends at~$e(y)$, and lies in $B_N(z_i,\epsilon_1(z_i))\cap e(M)$.
This segment is unique up to parametrization.

Let $\Gamma_z$ denote the vector $\frac{\mathrm{d}}{\mathrm{d}
s}\gamma_z(s)|_{s=0}$ tangent to $\gamma_z(s)$ at the point $z\in
e(B(y,\epsilon_M))\setminus e(\{y\})$. Here and in what follows, we
assume $\gamma_z(s)$ is parametrized by arc length. It is
well-understood that $\grad\dist_{N;e(y)}(z)$ must coincide with
$-\Gamma_z$. This fact yields the formula
\begin{equation}\label{grad is der}\grad\dist^2_{N;e(y)}(z)=-2\dist_{N;e(y)}(z)\Gamma_z.\end{equation}
We emphasize that~(\ref{grad is der}) holds when $z$ lies in
$e(B(y,\epsilon_M))\setminus e(\{y\})$.

Given $x\in B(y,\epsilon_M)\setminus\{y\}$, define the curve segment
$\tau_x(s)$ in the manifold $M$ by setting
$\tau_x(s)=e^{-1}(\gamma_{e(x)}(s))$. The notation $\mathrm T_x$
refers to the vector $\frac{\mathrm{d}}{\mathrm{d}
s}\tau_x(s)|_{s=0}$ tangent to $\tau_x(s)$ at the point $x$. We have
the equality $\mathrm T_x=(\mathrm{d} e)^{-1}\Gamma_{e(x)}$.
Together with~(\ref{dist eq}), (\ref{grad dist sq}), and~(\ref{grad
is der}), this implies
\[\grad\dist^2_{y}(x)=-2\dist_{y}(x)\mathrm T_x.\]
If $x$ lies in $\partial M$, then the vector $\mathrm T_x$ satisfies
$\left<\mathrm T_x,\nu\right>\le0.$ Consequently,
\[\left<\grad\dist^2_{y}(x),\nu\right>\ge0.\]
Our arguments prove this for
$x\in(B(y,\epsilon_M)\setminus\{y\})\cap\partial M$. If $y$ belongs
to $\partial M$, then the estimate can be extended to $y$ by
continuity. Hence the desired nonnegativity of
$\frac\partial{\partial\nu}\dist^2_{y}(x)$.

We will now establish Statement~\ref{lem dist st3} of the lemma. A
calculation based on~(\ref{dist eq}) shows that
\begin{equation}\label{DM dist DN dist}
\Delta_M\dist^2_{y}(x)=\left.\Delta_N\dist^2_{N;e(y)}(z)\right|_{z=e(x)}
\end{equation}
when $x\in B(y,\epsilon_M)$. Here, $\Delta_N$ denotes the
Laplace-Beltrami operator on~$N$. Our intension is to estimate the
right-hand side of Eq.~(\ref{DM dist DN dist}) using Theorem~(2.28)
in~\cite{AK82}; cf.~\cite[Section~3.4]{EH02b}.

Let us lay down a few preliminary facts. Consider a point $z\in
e(B(y,\epsilon_M))\setminus e(\{y\})$. As proved above, one can join
$z$ with $e(y)$ by the minimizing geodesic segment $\gamma_z(s)$. We
should point out that this segment is entirely contained in $e(M)$.
It is convenient to assume $\gamma_z(s)$ is parametrized by arc
length. Choose a constant $K>0$ satisfying the formula
\[
-(\dim M-1)K^2\le\inf\Ric(X,X).
\]
The infimum is taken over all the vectors $X\in TM$ with
$\left<X,X\right>=1$. It is finite because $M$ is compact. The
following assertion is easy to verify: At every point of the segment
$\gamma_z(s)$, the Ricci curvature of $N$ in the direction
$\frac{\mathrm{d}}{\mathrm{d} s}\gamma_z(s)$ is greater than or
equal to $-(\dim M-1)K^2$.

We are now ready to estimate $\Delta_N\dist^2_{N;e(y)}(z)$ by means
of Theorem~(2.28) in~\cite{AK82}. Note that the manifold $N$ is not
necessarily complete. Therefore, it is essential to take account of
the remark following Theorem~(2.31) in~\cite{AK82}. As mentioned in
the previous paragraph, one can join $z$ with $e(y)$ by the segment
$\gamma_z(s)$. The Ricci curvature of $N$ in certain directions is
bounded below by $-(\dim M-1)K^2$. With these facts at hand,
Theorem~(2.28) from~\cite{AK82} implies
\begin{align}\label{comp thm ineq}
\Delta_N&\dist^2_{N;e(y)}(z) \nonumber \\ &\le2(\dim
N-1)K\dist_{N;e(y)}(z) \coth\left(K\dist_{N;e(y)}(z)\right)+2.
\end{align}
We emphasize that~(\ref{comp thm ineq}) holds when $z$ belongs to
$e(B(y,\epsilon_M))\setminus e(\{y\})$. For a relevant inequality,
see~\cite[Corollary~3.4.4]{EH02b}.

Only a few simple remarks are now needed to finish the proof. Note
that the function $\dist_{N;e(y)}(z)$ takes its values in the
interval $(0,\epsilon_M)$ when $z$ varies through
$e(B(y,\epsilon_M))\setminus e(\{y\})$. Define the constant
$K_{\epsilon_M}>0$ by setting
\[
K_{\epsilon_M}=2(\dim N-1)K\sup_{r\in(0,\epsilon_M)}(r\coth(Kr))+2.
\]
In view of~(\ref{DM dist DN dist}) and~(\ref{comp thm ineq}), we
must have
\[
\Delta_M\dist^2_{y}(x)\le K_{\epsilon_M},\qquad x\in
B(y,\epsilon_M)\setminus\{y\}.
\]
This estimate extends to the point $y$ by continuity. Hence the
desired result.  \end{proof}

\begin{remark}
If $M$ were a closed manifold, then Statements~\ref{lem dist st1}
and~\ref{lem dist st3} of Lemma~\ref{lemma dist} would hold for
every constant $\epsilon_M$ less than or equal to the injectivity
radius of~$M$; see, for example,~\cite[Section~3.4]{EH02b}.
\end{remark}

We are now ready to establish Proposition~\ref{proposition exit
time}. Our line of reasoning is borrowed from~\cite[Proof of
Lemma~4.1]{MARBAT02}. In particular, we make use of Bernstein's
inequality related to local martingales.

\begin{proof}[Proof of Proposition~\ref{proposition exit time}.]
Introduce the process $N^y_t=\dist_y^2(X^y_t)$. Fix a constant
$\epsilon_M>0$ satisfying Statements~\ref{lem dist st1},~\ref{lem
dist st2}, and~\ref{lem dist st3} of Lemma~\ref{lemma dist}. Denote
$\epsilon_0=\frac12\min\{1,\epsilon_M\}$. Given a number
$r\in(0,1)$, consider the hitting time
\[\upsilon=\inf\left\{t\ge0\,\left|\,N_t^y=\epsilon_0^2r^2\right.\right\}.\]
It is easy to see that
\begin{align}\label{ext tm aux3}
\mathbb P\left\{\tau(y,r)<\kappa r^2\right\}&=\mathbb
P\left\{\sup_{t\in[0,\kappa r^2)} N_t^y>r^2\right\} \nonumber \\
&\le\mathbb P\left\{\sup_{t\in[0,\kappa r^2]}
N_{t\wedge\upsilon}^y\ge\epsilon_0^2r^2\right\}
\end{align}
for all $\kappa\in(0,\infty)$. Let us estimate the rightmost
probability in this formula.

By virtue of~(\ref{Ito u_t}) and~(\ref{lift of operators}), the
process $N_{t\wedge\upsilon}^y$ satisfies
\begin{equation}\label{ext tm aux1}
N_{t\wedge\upsilon}^y=\Upsilon_t+\frac12\int_0^{t\wedge\upsilon}\Delta_M\dist^2_y\left(X_s^y\right)\mathrm{d}
s
-\int_0^{t\wedge\upsilon}\frac{\partial}{\partial\nu}\dist^2_y\left(X_s^y\right)\mathrm{d}
L_s.
\end{equation}
The notation $\Upsilon_t$ refers to the local martingale
\[
\sum_{i=1}^n\int_0^{t\wedge\upsilon}(\mathcal
H_il)\left(u_s^Y\right)\mathrm{d} B_s^i
\]
with $l(u)=\dist^2_y(\pi(u))$ when $u\in O(M)$. Lemma~\ref{lemma
dist} implies that the second term in the right-hand side
of~(\ref{ext tm aux1}) is bounded above by $\frac12K_{\epsilon_M}t$
and the third term is nonnegative. As a consequence, the estimate
\begin{equation}\label{ext tm aux2}
\mathbb P\left\{\sup_{t\in[0,\kappa r^2]}
N^y_{t\wedge\upsilon}\ge\epsilon_0^2r^2\right\}\le\mathbb
P\left\{\sup_{t\in[0,\kappa r^2]}
\left(\Upsilon_t+\frac12K_{\epsilon_M}t\right)\ge\epsilon_0^2r^2\right\}
\end{equation}
holds for all $\kappa\in(0,\infty)$.

We now set $\kappa_0=\frac{\epsilon_0^2}{K_{\epsilon_M}}$ and assume
$\kappa\in(0,\kappa_0)$. Then the probability in the right-hand side
of~(\ref{ext tm aux2}) cannot exceed
\[
\mathbb P \left\{\sup_{t\in[0,\kappa
r^2]}\Upsilon_t\ge\frac12\epsilon_0^2r^2\right\}.
\]
Define $\eta=\frac{\epsilon_0^2}{32}$. As a computation shows, the
quadratic variation $\langle\Upsilon,\Upsilon\rangle_t$ of the local
martingale $\Upsilon_t$ satisfies
\begin{align*}\langle\Upsilon,\Upsilon\rangle_t&=\sum_{i=1}^n\int_0^{t\wedge
\upsilon}(\mathcal H_il)^2\left(u_s^Y\right)\mathrm{d} s \\
&=4\int_0^{t\wedge \upsilon}\dist^2_y(X_s^y)\,\mathrm{d}
s\le4\epsilon_0^2r^2t=128\eta r^2t.
\end{align*} In accordance with Bernstein's
inequality (see, for example,~\cite[Exercise~(3.16) in
Chapter~IV]{DRMY99}), the fact that
$\langle\Upsilon,\Upsilon\rangle_t$ cannot exceed $128\eta r^2t$
implies
\[
\mathbb P \left\{\sup_{t\in[0,\kappa
r^2]}\Upsilon_t\ge\frac12\epsilon_0^2r^2\right\}=\mathbb P
\left\{\sup_{t\in[0,\kappa r^2]}\Upsilon_t\ge16\eta
r^2\right\}\le\e^{-\frac{\eta}{\kappa}}.
\]
Combining this estimate with~(\ref{ext tm aux3}) and~(\ref{ext tm
aux2}) completes the proof.  \end{proof}

Proposition~\ref{proposition exit time} may be important to the
further development of the probabilistic approach to the Yang-Mills
heat equation on manifolds with boundary. In particular, this result
helps us obtain the following estimate for the curvature of the
time-dependent connection $\nabla(t)$ discussed in Section~\ref{sec
YMH}. As before, we deal with the reflecting Brownian motion $X_t^y$
starting at the point $y$. The connection $\nabla(t)$ solves
Eq.~(\ref{YM heat eq}) with the boundary conditions~(\ref{relative
BC}) or~(\ref{absolute BC}).

\begin{theorem}\label{thm a priori}
Let $\partial M$ be totally geodesic. Suppose either
Assumption~\ref{assu curv bndy} of Theorem~\ref{theorem LYH no crv}
or Assumptions~\ref{assu Ricci par} and~\ref{assu nneg sect} of
Theorem~\ref{theorem LYH crv} are fulfilled for $M$. Then there
exist constants $\xi_1>0$ and $\theta>0$, a function $\sigma(s)$ on
$(0,\infty)$, and a function $v(s)$ on $(0,1)$ that depend on
nothing but $M$ and satisfy the following statements:
\begin{enumerate}
\item
The values of $\sigma(s)$ and $v(s)$ lie in $(0,1)$ and
$(0,\infty)$, respectively.
\item
The function $\sigma(s)$ is non-increasing, while $v(s)$ is
non-decreasing.
\item
Given $s_0\in(0,T)$, $a\in(0,1]$, and
$s\in\left(0,\min\left\{\sigma\left(a^{-1}\YM(0)\right),s_0\right\}\right]$,
if
\[
s^2\mathbb
E\left(\left|R^{\nabla(s_0-s)}(X^y_s)\right|_E^2\right)\le a\xi_1,
\]
then the estimate
\begin{equation}\label{a priori est}
\sup_{t\in\left[s_0-(v(s))^2,s_0\right]}\sup_{x\in\bar
B(y,v(s))}\left|R^{\nabla(t)}(x)\right|_E^2\le\frac{a\theta}{(v(s))^4}
\end{equation}
holds true.
\end{enumerate}
\end{theorem}

In order to establish Theorem~\ref{thm a priori}, we need to repeat
the arguments from~\cite[Proof of Theorem~4.2]{MARBAT02}. Let us
outline the changes required for these arguments to work on a
manifold with boundary. The Brownian motion $X_t(x^*)$
in~\cite{MARBAT02} must be replaced by a reflecting Brownian motion
starting at~$x^*$. The stochastic differential equation for the
process $Y_s$ in~\cite{MARBAT02} then comes from Eqs.~(\ref{Ito
u_t}) and~(\ref{lift of operators}) of the present paper; cf.~the
proof of Theorem~\ref{thm YMH dim23}. The necessary estimates are
provided by Lemma~\ref{lemma Neumann norm}, Lemma~\ref{lemma monot},
and Proposition~\ref{proposition exit time}. We will not discuss
further details here.

\begin{remark}
Theorem~4.2 in~\cite{MARBAT02} establishes~(\ref{a priori est}) on a
closed manifold. Section~4 of~\cite{MARBAT02} contains a variety of
corollaries of this estimate. Many of those results would likely
generalize to manifolds with boundary by means of our
Theorem~\ref{thm a priori}.
\end{remark}

\begin{remark}
The paper~\cite{MARBAT02} uses its version of~(\ref{a priori est})
to prove that the curvature of the corresponding solution of the
Yang-Mills heat equation does not blow up if the dimension is less
than~4 or if the initial energy is small. Such a line of reasoning
may be inefficient on manifolds with boundary. For example,
Theorems~\ref{thm YMH dim23} and~\ref{thm YMH dim4} cannot be
deduced from Theorem~\ref{thm a priori} because their assumptions
are considerably weaker. The case where $\dim M$ is~5 or higher is
different. It seems likely that a statement similar to
Theorem~\ref{thm YMH dim5+} can indeed be obtained as a consequence
of Theorem~\ref{thm a priori}.
\end{remark}

\section*{Acknowledgements}

I express my profound gratitude to Prof.~Leonard Gross for his
support and numerous productive discussions. I am also thankful to
Prof.~Xiaodong Cao for the helpful conversations about the geometric
side of the paper.

\end{document}